%% file: main.tex
\documentclass[10pt,english, oneside,reqno]{smfart}

\usepackage[T1]{fontenc}
\usepackage[english]{babel}
\usepackage[utf8]{inputenc}

\usepackage[dvipsnames]{xcolor}
\usepackage{appendix}
\usepackage{lipsum}
\usepackage{amsmath,amssymb}
\usepackage{amsfonts}
\usepackage{smfthm}
\usepackage{mathrsfs}
\usepackage{systeme}
\usepackage{stmaryrd}
\usepackage{pgfplots}
\usepackage{minted}
\usepackage{enumitem}
\usepackage{nicematrix}
\usepackage{todonotes}
\usepackage{mathabx}
\usepackage{mathtools}
\usepackage{centernot}
\usepackage{verbatim}
\usepackage{stmaryrd}

\newcommand{\NN}{\mathbb{N}}
\newcommand{\RR}{\mathbb{R}}

\renewcommand{\tilde}{\widetilde}

\newcommand{\vare}{{\varepsilon}}
\newcommand{\vphi}{{\varphi}}
\newcommand{\hatsig}{{\tilde{\sigma}}}

\DeclareMathOperator{\sech}{sech}
\DeclareMathOperator{\supp}{Supp}

\usepackage{lastpage}
\usepackage{blindtext}
\usepackage{algorithm}
\usepackage{algpseudocode}
\usepackage{graphicx}
\usepackage{wrapfig}
\usepackage{tikz}

\usepackage[top=2.5cm,bottom=2.5cm,margin=2.5cm]{geometry}
\usepackage{fancyhdr}
\usepackage{hyperref}
\hypersetup{colorlinks=true, linkcolor=black, filecolor=black, citecolor = black, urlcolor=gray,}
\usepackage{tikz}
\usetikzlibrary{patterns}
\usepackage{calrsfs}
\usepackage{systeme}
\usepackage{amsfonts}
\usepackage{stmaryrd}
\usepackage{amsthm}
\usepackage{mathtools}
\usepackage{amsmath}
\usepackage{dsfont} 
\usepackage{amssymb}

\usepackage{lipsum}

\usepackage{graphicx}
\graphicspath{ {./images/} }
\newtheorem{definition}{Definition}[section]
\newtheorem{theorem}{Theorem}[section]

\newtheorem{remark}{Remark}[section]
\newtheorem{lemma}[theorem]{Lemma}
\newtheorem{proposition}[theorem]{Proposition}

\usepackage{pgfplots}
\pgfplotsset{compat=1.5}
\usepackage{mathrsfs}
\usetikzlibrary{arrows}
\usepackage[
   backend=biber,        
   sorting=nyt,          
]{biblatex}
\usepackage{csquotes}
\usepackage{graphicx} 
\usepackage{tikz}
\usetikzlibrary{positioning}
\usepackage{inputenc}

\usepackage{emptypage}

\title{Slowly travelling infinite point blow-up \\ for the critical generalized KdV equation}
\author{Nailya Manatova
\\ \textsc{Laboratoire de mathématiques de Versailles}
\\ \textsc{Université de Versailles St-Quentin en Y., CNRS}
\\ \textsc{Université Paris-Saclay}
\\ nailya.manatova@uvsq.fr}

\begin{document}
\begin{abstract}
    We study the finite time blow up phenomenon for the quintic, mass critical gKdV equation.
    We prove the existence of a class of solutions $U$ with infinite point, finite time blow up behavior, at the particular blow up rate 
    $$\|\partial_x U(t)\|_{L^2} \sim (T-t)^{-\nu}\quad \text{as} \quad t \uparrow T,$$
    where $\nu = \frac 12$, $T$ is the blow up time and where the travel speed of the blow up bubble is logarithmic. Therefore, we call this behaviour slowly travelling infinite point blow up.
    The special blow up rate $\nu =\frac 12$ is a threshold which separates finite and infinite point bubbling. In a previous work \cite{1M}, the author constructed other infinite point blow up solutions for the continuum of rates $\nu\in(\frac 12,1)$, using polynomial tails in the space variable and extending the results in \cite{MMPIII}, restricted to $\nu> \frac{11}{13}$.
    However, that work suggested a change of the tail for the threshold case.
    In the present paper, we consider an exponentially decaying tail on the right in space.
    As in \cite{1M}, the initial data can be taken arbitrarily close to the ground state in $H^1$.
    
    From a technical perspective, in addition to the change of tail, we have to adapt the energy-virial functional to the presence of the exponential tail, by modifying a scaling term used to control the right-hand side of the blow up solution.
\end{abstract}

\maketitle

\input{Intro}

\input{Decaying_tail}

\input{Decomposition_around_Q} 
\input{Energy_estimates}

\input{Construction_of_exploding_solution}

\printbibliography
\end{document}

%% file: Intro.tex
\
\section{Introduction}


\subsection{Presentation of the problem}

We consider the $L^2$ critical generalized Korteweg-de Vries equation (gKdV)
\begin{equation}\tag{gKdV}\label{gKdV_principal_eq}
    \partial_t U + \partial_x(\partial_{xx}U + U^5)=0, \qquad (t,x)\in [0,T) \times\RR,
\end{equation}
where $U(t,x)$ is a real valued function. In the following, we write $H^1$ instead of $H^1(\RR)$
and similarly for other spaces of functions.

The Cauchy problem is locally well-posed in the energy space $H^1$ from the works \cite{Kato-1983}, \cite{KPV93}. For a given $U_0 \in H^1$, there exists a unique maximal solution $U$ of \eqref{gKdV_principal_eq} in a functional space included in $C([0,T), H^1)$.
Moreover, \cite[Corollary 1.4]{Kato-1983} states that a solution $U$ whose maximal time of existence $T$ is finite has to blow up at $T$ and more precisely has to satisfy
\begin{equation}\label{intro_rule_out_of_less_that_frac13}
    \liminf_{t\uparrow T} (T-t)^{\frac 13} \|\partial_x U(t)\|_{L^2} > 0.
\end{equation}

For an $H^1$ solution, the mass and the energy, defined respectively by
\begin{equation}
    M(U)(t) = \int U^2 (t,x) \, dx, \qquad E(U)(t) = \frac 12 \int(\partial_x U)^2(t,x)\, dx - \frac 16 \int U^6(t,x)\, dx\,,
\end{equation}
are conserved along the evolution in time.
 
The equation is invariant under scaling and translation: if $U$ is a solution, 
then for any $\lambda>0$, $\sigma \in \RR$,
\begin{equation}
    U_{\lambda,\sigma}(t,x) := \lambda^{\frac 12}U(\lambda^3 t, \lambda x + \sigma)
\end{equation}
is also a solution to \eqref{gKdV_principal_eq}. The scaling symmetry leaves the $L^2$ norm invariant, so the problem is mass critical.

Recall that there exists a unique (up to translations) positive solution of the equation
\begin{equation}\label{equation_of_Q}
    -Q'' + Q - Q^5 =0 \quad \text{on}\; \RR
\end{equation}
which is given by
\begin{equation}
    Q(x) = \left( 3 \sech^2(2x) \right)^{\frac 14}.
\end{equation}

The family of solitary wave solutions of \eqref{gKdV_principal_eq} is given by 
\begin{equation}
    \lambda^{-\frac 12}_0 \,Q\big(\lambda^{-1}_0 (x - \lambda^{-2}_0 t - \sigma_0) \big) \quad\text{for}\quad  (\lambda_0,\sigma_0) \in (0,+\infty)\times \RR. 
\end{equation}

The ground state $Q$ satisfies $E(Q) =0$ and reaches the optimal constant in the sharp Gagliardo-Nirenberg inequality (see for instance \cite{Weinstein-1983})
\begin{equation}\label{gagliardo-nirenberg}
    \frac 13\int \phi^6 \leq \bigg( \frac{\int \phi^2}{\int Q^2} \bigg)^2\, \int (\partial_x \phi)^2, \qquad \forall \phi \in H^1.
\end{equation}
The Gagliardo-Nirenberg inequality and the mass-energy conservation yields the global existence and $H^1$ boundedness of the solution with $H^1$ initial data such that
\begin{equation}
    \|U_0\|_{L^2} < \|Q\|_{L^2}.
\end{equation}


\subsection{Related results on the finite time blow up of gKdV}
We focus on finite time blowing up solutions satisfying 
\begin{equation}\label{rate}
    \|  \partial_x U(t)\| \sim C (T-t)^{-\nu} \quad \hbox{as $t\to T$,}
\end{equation}
where $\nu > 0$ is called the \textit{blow up rate}.
Moreover, we are interested in blow up solutions with slightly supercritical mass initial data $U_0$, which means that for some $\delta>0$ small enough, it holds
\begin{equation}\label{10}
    \|Q\|_{L^2} \leq \|U_0\|_{L^2} < (1+\delta)\|Q\|_{L^2}.
\end{equation}

According to the result stated in \eqref{intro_rule_out_of_less_that_frac13}, the blow up rate must satisfy $\nu \geq \frac 13$. The work \cite{Martel-Merle-02} rules out the blow up rate $\nu = \frac 13$ when \eqref{10} holds. 
This work also proves that the solitary wave is the only attractor in this regime in the sense that the following asymptotic behavior holds
\begin{equation}\label{parametres}
    \lambda^{\frac 12}(t) U(t, \lambda(t) x + \sigma(t)) \rightharpoonup \pm Q \quad 
 \mbox{as $t\uparrow T$ in $H^1$ weak,}
\end{equation}
where the functions $\lambda$ and $\sigma$, respectively called translation parameter and scaling parameter satisfy in addition
\begin{equation}
    \lim_{t\uparrow T}\lambda(t) \|\partial_x U(t)\|_{L^2}=\|Q'\|_{L^2},\quad
    \lim_{t\uparrow T}\lambda^2(t)\sigma(t) = 1.
\end{equation}
Inserting \eqref{rate}, this shows that infinite point blow up corresponds to $\nu \geq \frac 12$ and finite point blow up occurs for $\frac 13 <\nu < \frac 12$. 

In the series of works \cite{MMPI}, \cite{MMPIII}, \cite{MMPII}, the authors studied the infinite point blow up regime more precisely. In particular, they have approached the problem by introducing an approximate blow up profile $Q_b = Q+b P_b$, where $b$ is a small additional parameter, instead of the solitary wave itself. They have also introduced a functional combining  Kato-type monotonicity properties of \eqref{gKdV_principal_eq} (see \cite{Kato-1983}),
energy estimates for the linearised problem around $Q$ (see \cite{Weinstein-85})
and virial-type identities (see \cite{Martel-Merle-00}). 
Summarizing, their work established the following facts:
\begin{enumerate}
    \item For the critical mass ($\|U_0\|_{L^2} = \|Q\|_{L^2} $), there exists a unique blowing-up solution $S(t)$ with the rate $\nu =1$. Furthermore, the behaviour of $S$ is known for all negative times.
    \item For slightly supercritical mass solutions, there exists a set $\mathcal{A}$
    of initial data (included in $H^1$), in which all the possible behaviours are classified. Bubbling necessarily occurs at the blow up rate $\nu =1 $ and is an open property in this topology. The other two possible behaviours are global solutions that converge to a final solitary wave, and an exit from the soliton neighbourhood by vanishing. 
    Moreover, the soliton behaviour is a $C^1$ co-dimension one manifold which separates the  other two cases. (See \cite{Nakanishi-16}.)
    \item There exists a large class of exotic blow up solutions with any rate $\nu >\frac{11}{13}$. Moreover, these solutions can be taken arbitrarily close to the ground state $Q$,  but do not belong to $\mathcal A$. Such solutions are constructed by perturbing an approximate blow up profile with a decaying polynomial tail on the right. 
    In \cite{MMPIII}, the decaying tail is of the polynomial form $x^{-\theta}$ where $\theta \in \big(1,\frac{29}{18}\big)$.
    Global, infinite time blow up solutions with $\|U(t)\|_{H^1}\sim C t^{\nu}$ with $t \uparrow +\infty$ for all $\nu>0$ are also constructed in this work.
\end{enumerate}

It turns out that the range $\nu >\frac{11}{13}$ found in \cite{MMPIII} was not optimal.
A recent work \cite{1M} has covered the blow up rates $\nu \in(\frac 12,1)$. There, the decaying tail was considered to be $x^{-\theta}$ with $\theta \uparrow +\infty$ as $\nu \downarrow \frac 12$ and the energy-virial functional had to be adapted to such potentially very decaying tails (when $\theta$ is large).
Moreover, the relation between $\nu$ and $\theta$ 
indicates that, in order to reach the threshold rate $\nu =\frac 12$, 
the decay behaviour on the tail must be changed for an exponential one. 
However, the energy estimates proved in \cite{1M} to construct the solution for
any polynomial tail do not extend to the threshold case. Indeed, the margin $\nu -\frac 12 >0$ allowed to compensate some terms in the energy estimates. 
In the case of $\nu = \frac 12$, this extra room is not available and one has to
further modify the energy-virial functional to address this issue.

To complete this short review on the blow up problem for \eqref{gKdV_principal_eq}, 
we observe that the first example of a finite point blow up solution was constructed in \cite{Martel-Pilod-2024} for the blow up rate $\nu = \frac 25$. More recently, the work \cite{Martel-Pilod-2026} constructed solutions with any blow up rate $\nu\in\big(\frac 37, \frac 12 \big)$, corresponding to a blow up residue of the form $x^{\alpha-\frac 12}$ for
$x>0$ close to $0$,  the blow up point,
where $\alpha \in (1,+\infty)$ is explicitly related to $\nu$.
The existence of finite point blow up in $H^1$ with rates in $\big(\frac 13, \frac 37 \big)\setminus \{\frac25\}$ remains an open question.

\subsection{Main result}
We present the main theorem on the existence of the slow infinite point blow up.

\begin{theorem}\label{Theorem_principal_result}
For any $\delta > 0$ there exists $0<T<+\infty$ and $U_0 \in H^1(\RR)$ with $\|U_0 - Q\|_{H^1} \leq \delta$ such that the solution $U$ of the \eqref{gKdV_principal_eq} evolving from $U_0$ exists on $[0,T)$ and blows-up at time $T$ with
\begin{equation}
    \|\partial_x U(t)\|_{L^2} \sim C (T-t)^{-\frac 12} \qquad \text{as} \quad t \uparrow T
\end{equation}
where $C>0$ is a constant.
Moreover, the travel parameter $\sigma(t)$ defined in \eqref{parametres}
for this solution satisfies 
\begin{equation}
        \sigma(t) \sim \frac 12 \ln\Big( \frac{1}{T-t}\Big) \quad \text{as}\quad t \uparrow T.
\end{equation}
\end{theorem}

\vspace{0.2cm}
\begin{remark}
We comment on further possible directions of investigation.
Together with \cite{MMPIII} and \cite{1M} the above result concludes the study of the infinite point blow up 
of the form \eqref{rate} for the mass critical \eqref{gKdV_principal_eq}. 
The study of non-polynomial blow up rates, such as oscillations between two polynomial rates
(see \cite{Donninger-Huang-Krieger-Schlag-2014})
or a given polynomial rate perturbed by a logarithm, is a possibility.
More generally, an extensive investigation on the correlation between a given decaying tail and the blow up behavior of a solution would be interesting
(see \cite{Gustafson-Nakanishi-Tsai-asympt-2008} for a related results).
\end{remark}

\subsection{Notation}

For $1\leq p\leq +\infty$ we denote the classical real-valued Lebesgue space by $L^p(\RR)$. We denote the $L^2$-scalar product by
\begin{equation}
    (f,g) = \int_{\RR} f(y)\,g(y)\, dy, \qquad\text{for}\quad f,g \in L^2(\RR).
\end{equation}

We introduce the generator of scaling symmetry in the space $L^2$
\begin{equation}\label{def_of_scal_sym}
    \Lambda f = \frac{1}{2}f + yf'
\end{equation}
and the linearised operator $\mathcal{L}$ around the ground state
\begin{equation}\label{def_of_lin_operator}
    \mathcal{L}f = -f''+f-5Q^4 f.
\end{equation}

We introduce the weighted $L^2$ norms 
\begin{equation}\label{def_of_L2loc_L2B}
    \|f\|_{L^2_{sol}} = \Big( \int_{\RR} f^2(y) e^{-\frac{|y|}{10}}dy\Big)^{\frac{1}{2}} \qquad \text{and} \qquad \|f\|_{L^2_B}= \Big( \int_{\RR} f^2(y) e^{\frac{y}{B}} dy \Big)^{\frac{1}{2}}.
\end{equation}
where $B>100$ is to be fixed later.

For simplicity of notation, $dy$ is omitted inside of the integration and the integration domain is $\RR$ unless otherwise stated.

The inequality $f\leq C g$ with some time-space independent constant $C$ is abbreviated as $f\lesssim g$. In general terms, the constants which value is not crucial for the proof are denoted $C$ and their value is time-space independent.

For a given small positive constant $0<\alpha^* \ll 1$, we denote $\delta$ a small parameter depending on $\alpha^*$ such that
$ \delta(\alpha^*) \to 0$ as $\alpha^* \to 0$.

For some given interval $I \in \RR$, we denote $\mathbf{1}_{I}$ the characteristic function of this interval.

We use the abbreviation r.h.s for right-hand side.

\subsection{Strategy of the proof}\label{S:strategy}
The proof is divided into four main steps detailed in Sections \ref{S:2} to \ref{section_construction}.
\begin{enumerate}
    \item \textit{Definition of the tail}\\
    Given $c_0 < 0$, we define 
    \begin{equation}
        f_0(t,x) = c_0\, e^{-\frac 12 (x - \frac t4)} \qquad \text{for} \quad x -\frac t4 \gg 1.
    \end{equation}
    The time parameter is introduced to ensure that the decaying tail $f_0$ is a solution to the linear Airy equation for $x-\frac t4$ large. The solution $f$ of \eqref{gKdV_principal_eq} evolving from $f_0(t=0)$ is such that for all $k \in \mathbb{N} \cup\{0\}$, it holds
    \begin{equation}
        \big|\partial^k_x f(t,x) - \partial^k_x f_0(t,x)\big| \lesssim e^{-2x}
        \qquad \text{for} \quad x \gg t.
    \end{equation}
    In particular, $f_0$ is a good approximation of the solution for $x$ large.
    Note that, in contrast with the case of a polynomial tail, there is no gain
    when taking spatial derivatives.
    
   Another difference with the case of a polynomial tail
   is that the time dependence of the tail will also appear within the modulation system of the blow up parameters.
    
    \item \textit{Modulation around a refined profile and system of parameters}\\
    We look for a solution $U$ of \eqref{gKdV_principal_eq} of the following form
    \begin{equation}\label{intro_strat_proof_decomp_of_U}
        U(t,x) =\lambda^{-\frac 12}(s)\big( Q_{b(s)}(y) + f(t(s),\sigma(s))R(y) + \vare(s,y)\big) + f(t,x)
    \end{equation}
 where $(s,y)$ are the rescaled time and space variables defined by
    \begin{equation}
        \frac{ds}{dt} = \frac{1}{\lambda^3},\quad y = \frac{x-\sigma(s)}{\lambda(s)},
    \end{equation}
 the function $\varepsilon(s,y)$ is a small function to be controlled and
 $\lambda$, $\sigma$ and $b$ are time-dependent functions to be fixed.
   Note that as $t\in [0,T)$, $s$ is defined on a certain interval $[s_0,+\infty)$. 
   We define $\tau(s)$ as the inverse of the function $s(t)$.
    
    To account for the time dependence of the decaying tail, the law of the position parameter $\sigma$ also depends on the time variable $\tau(s)$.
    To simplify the notation, we introduce a new parameter
    \begin{equation}
        \hatsig(s) = \sigma(s) - \frac{\tau(s)}{4}.
    \end{equation}
    This parameter appears explicitly in the modulation system of the blow up law
    \begin{equation}\label{formal_dynamical_system}
        \frac{\lambda_s}{\lambda}+b = 0, \qquad \hatsig_s = \lambda, \qquad \frac{d}{ds}\Big( \frac{b}{\lambda^2} + \frac{4}{\int Q}\,c_0 \,\lambda^{-\frac{3}{2}}e^{-\frac12 \hatsig} \Big) = 0.
    \end{equation}
    Integrating the last equation on $[s_0,s]$, we obtain 
    \begin{equation}\label{equation_with_l0}
        \frac{b(s)}{\lambda^2(s)} + \frac{4\,c_0}{\int Q}\,\lambda^{-\frac 32}(s)\,e^{-\frac12 \hatsig(s)} = l_0 .
    \end{equation}
    For now on, we fix $l_0 = 0$ to simplify.
    Replacing $b$ by $-\frac{\lambda_s}{\lambda}$ and then $\lambda$ by $\hatsig_s$ (from \eqref{formal_dynamical_system}) yields
    \begin{equation}
        \lambda^{-\frac 12}\lambda_s  = \frac{4\,c_0}{\int Q}\, \lambda\,e^{-\frac12 \hatsig} = \frac{4\,c_0}{\int Q}\,\hatsig_s\,e^{-\frac12 \hatsig}.
    \end{equation}
    We look for $\lambda(s)\to 0^+$ as $s\to +\infty$ and thus $\lambda_s\leq 0$. Therefore the first equality leads us to choose $c_0<0$.\\
    After integration, we get 
    \begin{equation}\label{l1}
        \lambda^{\frac 12}(s) + \frac{4}{\int Q}c_0\,e^{-\frac12 \hatsig(s)} = l_1.
    \end{equation}
   Since $\lambda(s)\to 0$ and $\sigma(s)\to +\infty$ as $s\to +\infty$, one has to impose $l_1=0$.
   In the setting of the finite time blow up, $\tau(s)$ is bounded, therefore $\hatsig \to +\infty$ when $\sigma(s)\to +\infty$.
    Combining with the second equation in \eqref{formal_dynamical_system}, and fixing the initial value to be $e^{\hatsig(s_0)} = (4c_0/\int Q)^2 s_0$, we get 
    \begin{equation}
        \hatsig_s = \Big( \frac{4}{\int Q}c_0\Big)^2 e^{-\hatsig} \quad \iff \quad e^{\hatsig} = \Big( \frac{4}{\int Q} c_0\Big)^2\,s
    \end{equation}
    Inserting the expression of $\hatsig$ in the relations  \eqref{l1} and the equation of $b$ in \eqref{formal_dynamical_system} and fixing
    \begin{equation}\label{def_of_c_0}
        c_0 = -\frac{\int Q}{4},
    \end{equation}
    we get
    \begin{equation}\label{sketch_proof_formal_law_of_param}
        \lambda(s) = s^{-1},\qquad e^{\hatsig(s)} = s, \qquad b(s)= s^{-1}.
    \end{equation}

    \item \textit{Energy estimates}\\
    The decomposition in \eqref{intro_strat_proof_decomp_of_U}  yield the following behaviour of the correcting term 
    \begin{equation}
        \vare_s \sim \big(-\varepsilon_{yy} + \varepsilon- (\varepsilon^5 + 5(W+F)^4 \varepsilon) \big)_y + \frac{\lambda_s}{\lambda}\Lambda \varepsilon + \Big(\frac{\sigma_s}{\lambda}-1 \Big)\varepsilon_y.
    \end{equation}
    As in the references \cite{1M} and \cite{MMPIII}, we introduce a mixed energy-virial functional $\mathcal{F}$ and use its coercivity (up to imposing suitable orthogonality conditions on $\varepsilon$) to control the weighted $H^1$ norm of the solution. 
    
    In \cite{1M} the author introduced a new energy-virial functional $\mathcal{H}$ with a term at the scaling level  to compensate the \textit{bad sign} term arising from $-\frac{\lambda_s}{\lambda}\Lambda \vare$, where $\frac{\lambda_s}{\lambda}\leq 0$ in the considered regime. 
   (This allowed to replace and improve a two-step argument in \cite{MMPI} and \cite{MMPIII}.)
    The decaying tail was considered with a polynomial decay on the right of the form $x^{-\theta}$, where $\theta \uparrow +\infty$ as $\nu \downarrow \frac 12$. Due to this decay, the scaling term was taken as $\int_{y>0} (\lambda y)^k \vare^2$ with $k$ depending on $\theta$ and therefore on $\nu$ and $k \uparrow +\infty$ as $\nu \downarrow \frac 12$.
    
    In the present paper, we consider a corresponding term
    in the energy functional of the form
    \begin{equation}
        s^{\kappa}\int_{y>(\theta \lambda)^{-1}} e^{2\kappa \lambda y} \,\vare^2
    \end{equation}
    where $\kappa$ and $\theta$ are chosen to be large enough in order to estimate the \textit{bad sign} term and ensure the control of $\mathcal{H}$ through time.
    
    \item \textit{Construction of blowing up solutions}\\
    By a bootstrap argument,  we then construct a solution existing on the time interval $[s_0,+\infty)$, which corresponds to $[0,T)$ in the time variable $t$. The law of the rescaled variable yields
    \begin{equation}
        T-t = \int^{+\infty}_{s(t)} \frac{d\tau(s')}{ds'}\, ds' = \int^{+\infty}_{s(t)}\lambda^{3}(s')\,ds' \sim \frac 12 s^{-2}.
    \end{equation}
    The parameter $\lambda^{-1}$ represents the behaviour of the $H^1$ norm in the finite time blow up regime. We consider the passage back into the variable $(t,x)$, we write
    \begin{equation}
        \|\partial_x U(t)\|_{L^2} \sim \|Q'\|_{L^2}\lambda^{-1}(s(t)) \sim C (T-t)^{-\frac 12} \quad \text{as }\quad t \uparrow T,
    \end{equation}
    which indicates the desired behaviour of the bubbling. Furthermore, the travel parameter $\sigma$ satisfies the following behaviour when $t\uparrow T$
    \begin{equation}
        \sigma(t) \sim \frac 12 \ln\Big( \frac{1}{T-t}\Big).
    \end{equation}
    Therefore, we call $U$ a slow infinite point blow up solution.
    
\end{enumerate}

\vspace{0.4cm}
\textbf{Acknowledgments: }

\textit{(Funding: This work was supported by the University of Versailles Saint-Quentin-en-Yvelines and by the University Paris-Saclay.)
}

\textit{The author is deeply indebted to Xu Yuan (Chinese Academy of Sciences) for highlighting a flaw in an earlier version of the proof.}

%% file: Decaying_tail.tex
\section{Decaying tail on the right in space}\label{S:2}

Let $\Theta:\RR\to[0,1]$ be a smooth, non-decreasing function, such that
$\Theta|_{(-\infty, \frac{1}{4})} \equiv 0$ and $\Theta|_{(\frac{1}{2},+\infty)} \equiv 1$.
For $c_0<0$ fixed in \eqref{def_of_c_0} and $x_0 > 1$, we consider
\begin{equation}
    f_0(t,x) = c_0\, e^{-\frac 12 \big(x -\frac t4\big)} \Theta \Big(\frac{1}{x_0}\Big( x -\frac t4\Big)\Big) \qquad \text{for}\quad (x,t)\in \RR\times [0,+\infty).
\end{equation}

Note that, for all $k\in \NN\cup\{0\}$ and all $t \geq0$, it holds
\begin{equation}\label{estimation_of_norms_of_f_0}
    \|f_0(t)\|_{L^{\infty}_x} + \|\partial^{k}_x f_0(t)\|_{L^{\infty}_x}\lesssim e^{-\frac{x_0}{8}} \qquad \text{and}\qquad \|\partial^{k}_x f_0(t)\|_{L^2_x}\lesssim e^{-\frac{x_0}{8}}.
\end{equation}
 
Let $f$  be the solution of
\begin{equation}\label{equation_of_f}
    \begin{cases}
        \partial_t f + \partial_x(\partial_{xx}f+ f^5) = 0,\quad (x,t)\in \RR\times [0,+\infty), \\
        f(0,x) = f_0(0,x).
    \end{cases}
\end{equation}
 
 We announce the property of the persisting exponential tail in a time-space region on the right. 
 \begin{lemma}\label{lemma_on_the_decaying_tail}
    For $x_0>1$ large enough, the solution $f$ of \eqref{equation_of_f} is global, satisfies $f\in C(\RR,H^k(\RR))$ for any $k\geq 0$ and $\|f\|_{L^{\infty}_{t}H^k_x} \lesssim \delta(x_0^{-1})$.
    In addition, for all $t \geq 0$ and all $x$ verifying
    \begin{equation}\label{dec_t_condition_on_x}
        \frac{3}{4} x > 110t ,
    \end{equation}
    it holds
    \begin{equation}\label{lemma_decaying_tails_estim_deriv_on_x_order_k}
        \forall k \in \NN\cup\{0\}, \qquad \big| \partial^k_x f(t,x) - \partial^k_x f_0(t,x) \big| \lesssim \delta(x^{-1}_0)\, e^{-2x},
    \end{equation}
    \begin{equation}\label{lemma_decaying_tails_estim_deriv_on_t}
        \lvert \partial_t f(t,x) \rvert \lesssim e^{-\frac 12 x} .
    \end{equation}
\end{lemma}

 \begin{proof}
 The first statement is a consequence of the well-posedness of the Cauchy problem, see
 \cite[Theorem 2.8]{KPV93} and \cite[Corollary 2.9]{KPV93}.
 
 Now, we define
     \begin{equation}
         d(t,x) = f(t,x) - f_0(t,x),
     \end{equation}
    which satisfies the equation
    \begin{equation}\label{equation_of_g_proof_dec_tails}
        \begin{cases}
            \partial_t d + \partial_x \big( \partial^2_x d + (d+f_0)^5 - f^5_0 \big) = F_0,\\
            d(0,x) = 0,
        \end{cases}
    \end{equation}
    where
    \begin{equation}
        F_0 = - \partial^3_x f_0 - \partial_x(f^5_0)-\partial_t f_0.
    \end{equation}
    For all $k \in \NN\cup\{0\}$, $F_0$ satisfies for $(x,t)\in \RR\times[0,+\infty)$,
    \begin{equation}\label{estimation_on_F_0_all_deriv}
        |\partial^{k}_x F_0 (t,x)| \lesssim e^{-\frac 12 \big( x -\frac t4\big)}\mathbf{1}_{[\frac{x_0}{4},\frac{x_0}{2}]}\big( x -\frac t4 \big) + e^{-\frac 52 \big( x -\frac t4\big)}\mathbf{1}_{[\frac{x_0}{2},+\infty)}\big(x-\frac t4\big).
    \end{equation}
    
    Using $\|f\|_{L^{\infty}_t H^k_x} \lesssim \delta(x^{-1}_0)$ and \eqref{estimation_of_norms_of_f_0}, observe that, for all $k \in \NN \cup\{0\}$, it holds
    \begin{equation}\label{proof_dec_t_bound_norms_of_g}
        \|\partial^{k}_x d\|_{L^{\infty}_tH^k_x} \lesssim \delta(x^{-1}_0).
    \end{equation}
    
    \vspace{0.4cm}
    Let $k\in \NN \cup \{0\}$ and define for $t\geq 0$,
    \begin{equation}
        M_k(t) = \int ( \partial^k_x d )^2(t,x)\,e^{z(t,x)}\,dx\,\quad \text{with}\quad z(t,x) = \alpha x - \beta t ,
    \end{equation}
    where the parameters $\alpha$, $\beta$ are fixed as follows
    \begin{equation}\label{proof_dec_t_parameters_alpha_etc}
        \alpha = \frac{19}{4}, \qquad \beta = 110.
    \end{equation}
    We want to prove the following property of $M_k(t)$, for all $t\geq0$,
    \begin{equation}\tag{\text{HR}}\label{HR_on_k_dec_tails}
        M_{k}(t) + \frac 12 (\beta - \alpha^3)\int^t_0 \int (\partial^{k}_x d)^2 e^z  + \frac32\alpha \int^t_0 \int (\partial^{k+1}_x d)^2 e^{z}
        \lesssim  e^{\frac{\alpha-1}{2}x_0}.
    \end{equation}
    
    We will proceed by induction on $k$. 
    
    \vspace{0.4cm}
    First, we prove this property holds for $k = 0$. We compute
    \begin{equation}
    \begin{split}
        \frac{d}{dt}M_0
        &= 2\int d\, \partial_t d \, e^z + \int d^2 \, e^z \, \partial_t z\\ 
        &= 2\int d\, F_0\, e^z - 2\int d\,\partial_x \big[ \partial^2_x d + (d+f_0)^5 -f^5_0 \big]\, e^z -\beta \int d^2\,e^z.
    \end{split}
    \end{equation}
    The second term is rewritten by integration by parts
    \begin{equation}
    \begin{split}
        &- \int d\,\partial_x \big[ \partial^2_x d + (d+f_0)^5 -f^5_0 \big]\\
        &= \int (\partial_x d)(\partial^2_x d)e^z + \int (\partial_x d)\big( (d+f_0)^5 -f^5_0 \big) e^z + \alpha \int d(\partial^2_x d)e^z + \alpha\int d \big(  (d+f_0)^5-f^5_0\big)e^z\\
        &= -\frac{3\alpha}{2}\int (\partial_x d)^2 e^z  
         + \frac{\alpha^3}{2}\int d^2 e^z +\int (\partial_x d)\big(( d+f_0)^5 -f^5_0 \big) e^z + \alpha\int d \big( ( d+f_0)^5-f^5_0\big)e^z.
    \end{split}
    \end{equation}
    We treat the last two terms using the following identity
    \begin{equation}
        \frac 16 \partial_x \big[  (d+f_0)^6-f^6_0 - 6d f^5_0 \big] = (\partial_x f_0) \big[  (d+f_0)^5 -f^5_0 -5g f^4_0 \big] + (\partial_x d)\big(  (d+f_0)^5 -f^5_0 \big).
    \end{equation}
    We have 
    \begin{equation}
        \begin{split}
            &\int (\partial_x d)\big( (d+f_0)^5 -f^5_0 \big) e^z + \alpha\int d \big(  (d+f_0)^5-f^5_0\big)e^z \\
            &= \frac 16 \int \partial_x \big[  (d+f_0)^6 -f^6_0 -6d f^5_0 \big]e^z + \alpha \int d \big(  (d +f_0)^5- f^5_0\big)e^z - \int (\partial_x f_0)\big[  (d+f_0)^5 -f^5_0 -5d f^4_0\big]e^z\\
            &= - \frac 16 \alpha \int \big(  (d+f_0)^6-f^6_0\big) e^z + \alpha \int d  (d+f_0)^5 e^z - \int (\partial_x f_0)\big[  (d+f_0)^5 -f^5_0 -5d f^4_0\big]e^z.
        \end{split}
    \end{equation}
    Thus (note $\beta-\alpha^3>0$)
    \begin{equation}
    \begin{split}
        &\frac{d}{dt}M_0(t) + 3\alpha \int (\partial_x d)^2 e^z + (\beta-\alpha^3)\int d^2 e^z\\
        &=2\int d F_0 e^z  - \alpha \int \Big( \frac13  (d+f_0)^6-\frac 13 f^6_0 - 2d (d+f_0)^5\Big)e^z - 2 \int (\partial_x f_0)\big[  (d+f_0)^5 -f^5_0 -5d f^4_0\big]e^z.
    \end{split}
    \end{equation}
    The following Sobolev estimate holds true for $t\geq 0$,
    \begin{equation}\label{proof_dec_t_Sobolev1}
        \| d^4 e^{z}\|_{L^{\infty}_x} \lesssim \|d\|^2_{L^{\infty}_t L^2_x}\int \big( (\partial_x d)^2 + d^2 \big) e^z \lesssim \delta(x^{-1}_0)\int \big( (\partial_x d)^2 + d^2 \big) e^z.
    \end{equation}
    Therefore, by \eqref{proof_dec_t_bound_norms_of_g} it holds
    \begin{equation}
    \begin{split}
         \int \Big| \frac13  (d+f_0)^6-\frac 13 f^6_0 - 2d (d+f_0)^5\Big|e^z &\lesssim \int \big( d^6 + d^2 f^4_0 \big) e^z\\
         &\lesssim \| d^4(t) e^{z(t)}\|_{L^{\infty}_x}\int d^2 + \int d^2 f^4_0 e^z \lesssim \delta(x^{-1}_0)\int \big( d^2 + (\partial_x d)^2\big) e^z.
    \end{split}
    \end{equation}
    The bound $| d|^5 \lesssim d^6 + d^2$ yields
    \begin{equation}
        \int \Big| (\partial_x f_0)\big[  (d+f_0)^5 -f^5_0 -5d f^4_0\big] \Big| e^z \lesssim \int d^6 |f'_0|e^z + \int d^2 |f'_0||f_0|^3 e^z \lesssim \delta(x^{-1}_0) \int \big((\partial_x d)^2+ d^2\big) e^z.
    \end{equation}
    Finally, by Young inequality, we get
    \begin{equation}\label{proof_dec_tails_estim_deriv_on_M_0}
        \frac{d}{dt} M_0(t) + \frac32 \alpha \int (\partial_x d)^2 e^z + \frac 12 (\beta-\alpha^3)\int d^2 e^z \lesssim \int F^2_0 e^z.
    \end{equation}
    The term on the right hand side is estimated using the decay of $|F_0|$ in \eqref{estimation_on_F_0_all_deriv} and since $\alpha\in(1,5)$,
    \begin{equation}\label{proof_dec_tail_int_F_0}
    \begin{split}
        \int F^2_0 e^z \lesssim e^{\big( \frac{\alpha}{4}-\beta\big)t}  \Big( e^{\frac{\alpha-1}{2}x_0} + e^{\frac{\alpha-5}{2}x_0} \Big)
        \lesssim e^{-(\beta-\frac{\alpha}{4})t   + \frac{\alpha-1}{2}x_0}.
    \end{split}
    \end{equation}
    Therefore, by integrating the estimate \eqref{proof_dec_tails_estim_deriv_on_M_0} combined with \eqref{proof_dec_tail_int_F_0} yields that \eqref{HR_on_k_dec_tails} holds for $k=0$.
    
    \vspace{0.4cm}
    
    Assume \eqref{HR_on_k_dec_tails} holds for all $0\leq k'<k$, where $k'\in \NN \cup \{0\}$ and $k \in \NN$.
    
    By integration by parts the following estimate holds for all $l\in \NN\cup \{0\}$,
    \begin{equation}\label{proof_dec_tails_for_deriv_estim_on_dx_l}
        \|(\partial^l_x d)^2 e^{z}\|_{L^{\infty}} \lesssim \int \big[(\partial^l_x d)^2 + (\partial^{l+1}_x d )^2 \big]\,e^{z} \lesssim M_l + M_{l+1}.
    \end{equation}
    Thus, if $k'\geq 1 $
    \begin{equation}\label{consequence_of_HR_k}
        \eqref{HR_on_k_dec_tails}_{k'} \text{ and } \eqref{HR_on_k_dec_tails}_{k'-1} \Rightarrow \|(\partial^{k'-1}_x d)^2 e^{z}\|_{L^{\infty}} \lesssim  e^{\frac{\alpha-1}{2}x_0}.
    \end{equation}
    By integration by parts and the expression of $\partial_t d $ in \eqref{equation_of_g_proof_dec_tails}, we have
    \begin{equation}\label{proof_dec_tail_deriv_of_M_k}
        \begin{split}
            &\frac{d}{dt}\big[ M_k(t)\big] + (\beta - \alpha^3)\int (\partial^k_x d)^2 e^z + 3\alpha \int (\partial^{k+1}_x d)^2 e^z\\
            &= 2 \int (\partial^k_x F_0)(\partial^k_x d)e^z + 2 \int \partial^k_x \big(  (d+f_0)^5 - f^5_0 \big)\, \big[ (\partial^{k+1}_x d) + \alpha (\partial^k_x d) \big]e^z.
        \end{split}
    \end{equation}
    
    We estimate the purely non linear term on the right hand side. Other terms are controlled in a similar way, but more easily due to the presence of $f_0$ and \eqref{estimation_of_norms_of_f_0}.
    We use the Leibniz rule for the derivative of the product, it yields
    \begin{equation}
        \begin{split}
            \int \big(\partial^k_x ( d^5)\big)^2 e^z \lesssim \int \Big( \sum_{\substack{k_1+...+k_5 =k \\ k_i \leq k_{i+1}}} \prod\limits^5_{i=1}(\partial^{k_i}_x d) \Big)^2 e^z \lesssim \sum_{\substack{k_1+...+k_5 =k \\ k_i \leq k_{i+1}}} \prod\limits^4_{i=1}\|\partial^{k_i}_{x}d \|^2_{L^{\infty}}\int (\partial^{k_5}_x d)^2 e^z.
        \end{split}
    \end{equation}
    In the sum, if $k_4 = k$, then $k_5=0$, which is impossible. If $k_4 = k-1$, then $k_5= 1$ and therefore $k=1$ or $2$.
    For $k=1$ it creates a term $\|d\|^8_{L^{\infty}_x}\int(\partial_x d)^2 e^z$, which is compensated by the similar term with the negative sign in the expression of $\frac{d}{dt}M_1(t)$, due to the control of $d$ in \eqref{proof_dec_t_bound_norms_of_g}.
    For $k=2$, it sufficient to use \eqref{proof_dec_t_bound_norms_of_g} and apply
    $\eqref{HR_on_k_dec_tails}_{k=1}$ after the integration in time of $\frac{d}{dt}M_2(t)$.

    In all other cases we have $k_4 < k-1$. By \eqref{proof_dec_t_bound_norms_of_g} we have
    \begin{equation}\label{estimate_of_int_deriv_g^5}
    \begin{split}
        \int \big(\partial^k_x ( d^5)\big)^2 e^z 
        \lesssim \delta(x^{-1}_0) \sum^{k-1}_{k' =1}\int (\partial^{k'}_x d)^2 e^z  + \delta(x^{-1}_0)\int (\partial^k_x d)^2 e^z.
    \end{split}
    \end{equation}
    
    From \eqref{proof_dec_tail_deriv_of_M_k} by interpolation, combining with \eqref{estimate_of_int_deriv_g^5} and integrating in time (observe that $M_k(0) = 0$) and using $\eqref{HR_on_k_dec_tails}_{k'<k}$, it yields
    \begin{equation}
        \begin{split}
            &M_k(t) + \frac 12(\beta - \alpha^3)\int^t_0 \int (\partial^k_x d)^2 e^z + \frac12 3\alpha \int^t_0 \int (\partial^{k+1}_x d)^2 e^z \\
            &\lesssim \int^t_0 \int (\partial^k_x F_0)^2 e^z + \delta(x^{-1}_0)\sum^{k-1}_{k_5 =1}\int^t_0\int (\partial^{k_5}_x d)^2 e^z
            \lesssim  e^{ \frac{\alpha-1}{2}x_0} .
        \end{split}
    \end{equation}
    The term with $\partial^k_x F_0$ is estimated exactly as in \eqref{proof_dec_tail_int_F_0}. 
    
    The initiation step $k=0$ was proven previously. Therefore, by induction, $\eqref{HR_on_k_dec_tails}$ holds for all $k \in \NN \cup \{0\}$. 
    
    Hence, from \eqref{proof_dec_tails_for_deriv_estim_on_dx_l} and the choice of the parameters $\alpha,\beta$, we get for all $k \in \NN \cup \{0\}$ and for $x$ verifying \eqref{dec_t_condition_on_x}
    \begin{equation}
        |\partial^k_x d(t,x)|^2 \lesssim e^{-z}\, \big(M_k(t) + M_{k+1}(t) \big) \lesssim e^{-\frac{19}{4}x + 110t + \frac{15}{8}x_0} \lesssim e^{-\frac{x_0}{8}}e^{-4x},
    \end{equation}
    which proves \eqref{lemma_decaying_tails_estim_deriv_on_x_order_k}.
    
    Lastly, the estimate on the derivatives of $f$ and its equation yields for $x$ satisfying \eqref{dec_t_condition_on_x} 
    \begin{equation}
        |\partial_t f(t,x)| \lesssim |\partial^3_x f(t,x)|  + |f(t,x)|^4 |\partial_x f(t,x)| \lesssim e^{-\frac12 x},
    \end{equation}
    which proves \eqref{lemma_decaying_tails_estim_deriv_on_t}.
 \end{proof}

%% file: Decomposition_around_Q.tex
\section{Decomposition around the soliton}\label{section_decomposition} 

\subsection{Structure of the linearised operator}

We recall some properties of the linearised operator $\mathcal{L}$ around the soliton defined in \eqref{def_of_lin_operator}.

We first introduce the following function space
\begin{equation}
    \mathcal{Y} := \{\phi \in C^{\infty}(\RR,\RR) \,|\, \forall k \in \NN, \exists C_k, r_k>0 \;\text{ s.t. }\; \lvert \phi^{(k)}(y)\rvert\leq C_k(1+|y|)^{r_k}e^{-|y|}, \forall y \in \RR \}.
\end{equation}
We recall some standard results on the operator $\mathcal{L}$ (see for instance \cite{Weinstein-85} and \cite{Martel-Merle-01}).

\begin{lemma}[Properties of the linearised operator $\mathcal{L}$]\label{lemma_prop_of_linear_operator}
The operator $\mathcal{L}: H^2(\RR) \subset L^2(\RR)\rightarrow L^2(\RR)$ satisfies
\begin{enumerate}
    \item (Spectrum) The operator $\mathcal{L}$ has only one negative eigenvalue, $\mathcal{L}Q^3 = -8 Q^3$. Moreover, $\ker{\mathcal{L}} = \{aQ' : a \in \RR \}$ and $\sigma_{ess}(\mathcal{L})=[1,+\infty)$.
    \item (Scaling) $\mathcal{L}\Lambda Q = -2Q$ and $(Q,\Lambda Q)=0 $.
    \item (Coercivity) There exists $\nu_0>0$ such that for any $\phi\in H^1$
    \begin{equation}
        (\mathcal{L}\phi,\phi)\geq \nu_0\|\phi\|^2_{H^1}- \frac{1}{\nu_0}[(\phi,Q)^2 + (\phi, y\Lambda Q)^2 + (\phi, \Lambda Q)^2] .
    \end{equation}
    \item (Invertibility) There exists a unique even function $R \in \mathcal{Y}$ such that
    \begin{equation}\label{properties_of_R}
        \mathcal{L}R = 5Q^4, \qquad\qquad (Q,R) = -\frac{3}{4}\int Q.
    \end{equation}
    \item (Invertibility (bis))
    There exists a unique function $P \in C^{\infty}(\RR)\cap L^{\infty}(\RR)$ such that $P' \in \mathcal{Y}$ and 
    \begin{equation}
        (\mathcal{L}P)' = \Lambda Q, \qquad  \lim_{y \to -\infty}P(y) = \frac{1}{2}\int Q, \qquad \lim_{y\to+\infty}P(y)=0.
    \end{equation}
    Moreover, it holds
    \begin{equation}\label{relation_on_scalar_prod_P_and_Q}
        (P,Q) = c_1,  \qquad (P,Q') = 0 \quad\text{with}\quad c_1 = \frac{1}{16}\Big( \int Q \Big)^2 > 0.
    \end{equation}
\end{enumerate}
\end{lemma}

\begin{remark}
    By $\lim_{+\infty}P=0$ and since $P' \in \mathcal{Y}$, it holds
    \begin{equation}\label{property_on_right_on_P}
       | P(y)| \lesssim e^{-\frac{|y|}{2}}, \quad y>0.
    \end{equation}
\end{remark}

\subsection{Definitions and estimates of localized profiles}\label{Subsection:def_of_Q_b}
We introduce a one-parameter family of refined blow up profiles $b \in (0,+\infty) \mapsto Q_b$, where the parameter $b$ is chosen to be small enough. This family serves to describe the higher-order deformation of the ground state profile $Q_{b=0}$ in the blow up regime.

Introduce a function $\chi \in C^{\infty}(\RR)$ such that
\begin{equation}\label{def_of_chi}
    0\leq \chi\leq 1, \quad 0\leq (\chi'')^2 \lesssim \chi' \quad \text{on }\RR, \qquad \chi_{|(-\infty,-2)} \equiv 0 \quad \text{and} \quad \chi_{|(-1,+\infty)} \equiv 1.
\end{equation}

\begin{definition}[Localised profile]
    Let $\gamma = \frac{3}{4}$, the localised profile $Q_b$ is defined by
    \begin{equation}
        Q_b(y) = Q(y) + bP_b(y)
    \end{equation}
    with 
    \begin{equation}
        P_b(y) = P(y) \chi_b(y) \quad \text{and} \quad \chi_b(y) = \chi(b^{\gamma}y).
    \end{equation}
\end{definition}

The following lemma states the properties of localised profile $Q_b$.
\begin{lemma}[Approximate self-similar profiles $Q_b$]
    There exists $b^*>0$ small enough such that for all $0<b<b^*$, the following properties hold:
    \begin{enumerate}
        \item (Estimates on $Q_b$) For all $y\in \RR$,
        \begin{equation}\label{estimate_on_Q_b_and_its_derivs}
        \begin{split}
            \lvert Q^{(k)}_b(y)\rvert \lesssim e^{-|y|}+b \,e^{-\frac{|y|}{2}} + b^{1+k\gamma}\,\mathds{1}_{[-2,-1]}(b^{\gamma}y), \quad \forall k \geq 0.
        \end{split}
        \end{equation}

        \item (Equation of $Q_b$) Let the error term be defined by
        \begin{equation}
            -\Psi_b = \big( Q''_b-Q_b+Q^5_b \big)' + b\Lambda Q_b - 2 b^2 \frac{\partial Q_b}{\partial b}.
        \end{equation}
        Then, for all $y\in \RR$,
        \begin{equation}
            \lvert \Psi^{(k)}_b(y) \rvert \lesssim b^{1+(k+1)\gamma}\, \mathds{1}_{[-2,-1]}(b^{\gamma}y) + b^2 e^{-\frac{|y|}{2}}+ b^2 \mathds{1}_{[-2,0]}(b^{\gamma}y)\mathds{1}_{(k = 0)}, \qquad \forall k \geq 0.
        \end{equation}
        Moreover,
        \begin{equation}\label{projection_Psi_b_on_Y_and_norm_L2B}
            \lvert (\Psi_b,\phi)\rvert \lesssim b^2, \qquad \forall \phi \in \mathcal{Y} \qquad\text{and}\qquad \| \Psi^{(k)}_b (y) \|_{L^2_B} \lesssim C_B\, b^2, \qquad \forall k \geq 0.
        \end{equation}

        \item (Projection of $\Psi_b$ in the direction $Q$)
        \begin{equation}\label{projection_of_Psi_b_on_Q}
            \lvert (\Psi_b, Q) \rvert \lesssim b^3.
        \end{equation}

        \item (Mass and energy properties of $Q_b$)
        \begin{equation}\label{estimate_mass_of_Qb}
            \Big| \int Q^2_b - \Big( \int Q^2 + 2b ( P, Q) \Big) \Big| \lesssim b^{2-\gamma},
        \end{equation}
        \begin{equation}\label{estimate_energy_of_Qb}
            \Big| E(Q_b) + b (P, Q) \Big| \lesssim b^2.
        \end{equation}
    \end{enumerate}
\end{lemma}
\begin{proof}
The proofs of this results can be found in \cite[Lemma 2.4]{MMPI} and in \cite[Lemma 3.3]{Martel-Pilod}. 
\end{proof}

\subsection{Definition of the approximate solution}

Given time dependent functions $\lambda>0$, $\sigma$ in $C^1(\RR)$, we introduce for some $s_0$ large enough
\begin{equation}\label{def_of_rescaled_time}
    \tau(s) = \int^s_{s_0} \lambda^3(s')\,ds'.
\end{equation}
The new spatial variable is defined as follows
\begin{equation}
    y = \frac{x-\sigma(s)}{\lambda(s)}.
\end{equation}

For $U$ solution of \eqref{gKdV_principal_eq}, we pass to the rescaled variable by setting
\begin{equation}
    V(s,y) := \lambda^{\frac12}(s)\big( U(\tau(s),x) - f(\tau(s), x)\big), \qquad F(s,y) := \lambda^{\frac 12 }(s) f(\tau(s),x). 
\end{equation}
Recall that $f$ is the solution to \eqref{equation_of_f} evolving from $f_0$, introduced in the Section \ref{S:2}.

The error of \eqref{gKdV_principal_eq} in the rescaled variables is given by
\begin{equation}\label{def_of_error_term_mathcal_E}
        \mathcal{E}(V) = V_s+\partial_y \big( \partial^2_y V -V + (V+F)^5 - F^5 \big) -\frac{\lambda_s}{\lambda}\Lambda V - \Big( \frac{\sigma_s}{\lambda}-1 \Big)\partial_yV.
    \end{equation}
Hence, $U$ solves \eqref{gKdV_principal_eq} if and only if $\mathcal{E}(V) =0$.

Our aim is to construct an approximate solution $W$ of the equation $\mathcal{E}(V) = 0$ in the following form
\begin{equation}\label{def_of_W_and_r}
    W(s,y) = Q_{b(s)}(y)+r(s)R(y), \quad \text{with} \quad r(s) := F(s,0) = \lambda^{\frac 12}(s)f(\tau(s),\sigma(s)).
\end{equation}
Note that the interest of the rescaled variables is to study the soliton both standing and with constant scaling. 

The refined profile $Q_b$ was defined in Subsection \ref{Subsection:def_of_Q_b} and $b>0$ is another function of $s$ to be chosen. The term $r(s)R(y)$, where $R$ is introduced in Lemma \ref{lemma_prop_of_linear_operator}, aims to compensate the first-order interaction between the soliton and the perturbing tail $f$.

Fix 
\begin{equation}\label{rho_condition}
    0<\rho < \frac{1}{16}
\end{equation}
and set 
\begin{equation}
    \hatsig = \sigma - \frac14 \tau.
\end{equation}
The new parameter $\hatsig$ is introduced to take in account the time dependence of the tail $f_0$ (see Section \ref{S:strategy}).

We work under the following assumptions on the parameters, which are chosen in view of the formal asymptotics in \eqref{sketch_proof_formal_law_of_param}
\begin{equation}\tag{BS1}\label{BS1}
    \begin{cases}
        \lvert e^{\hatsig(s)} - s\rvert \;\leq\; s^{1-\rho} \,, \\
        \lvert \lambda(s) - s^{-1} \rvert \;\leq\; s^{-1-\rho} \,,\\
        \lvert b(s) - s^{-1} \rvert \;\leq\; s^{-1-\rho} \,.
    \end{cases}
\end{equation}

\begin{remark}\label{remark_on_the_compact_time_interval}
    Under the bootstrap condition $\eqref{BS1}$, the existence time of $U$ is finite. Therefore, for all $s \in [s_0,+\infty)$, the time variable $\tau(s)$ belongs to a compact set $[0,1]$. 
\end{remark}
\begin{remark}
In the following, we will need precise bounds on the parameters. Hence, we fix $s_0\gg1$ such that \eqref{BS1} implies that
\begin{equation}\label{conseq_of_BS1_strict_ineq}
    0<\frac{1}{2}s^{-1} \leq \lambda \leq \frac{3}{2}s^{-1}, \qquad \frac{1}{2}s \leq e^{\hatsig}\leq \frac{3}{2}s, \qquad 0<\frac{1}{2}s^{-1} \leq b \leq \frac{3}{2}s^{-1}.
\end{equation}
We also have 
\begin{equation}
    \frac 12 s \leq e^{\sigma} \leq \frac 32 e^{\frac 14} s.
\end{equation}
\end{remark}

\vspace{0.2 cm}
We introduce the auxiliary functions (related to the heuristic computations \eqref{equation_with_l0} and \eqref{l1})
\begin{equation}\label{definition-g-h}
    g(s)= \frac{b(s)}{\lambda^2(s)}+\frac{4}{\int Q}c_0\lambda^{-\frac{3}{2}}(s)e^{-\frac{1}{2}\hatsig(s)} \qquad\text{and}\qquad
    h(s) = \lambda^{\frac{1}{2}}(s) + \frac{4}{\int Q}c_0 e^{-\frac{1}{2}\hatsig(s)}.
\end{equation}
For simplicity of notation, we also define the following 
\begin{equation}\label{def_of_vec_M_m_n}
    \vec{m}= \begin{pmatrix}
        \frac{\lambda_s}{\lambda}+b\\
        \frac{\hatsig_s}{\lambda}-1
    \end{pmatrix}, \qquad \vec{n}= \begin{pmatrix}
        -b\\
        \frac{\lambda^2}{4}
    \end{pmatrix}, \qquad \vec{M}= \begin{pmatrix}
        \Lambda\\
        \partial_y
    \end{pmatrix}.
\end{equation}
\begin{remark}
Both $\vec{m}$ and $\vec{M}$ are classical, as they usually appear when computing $\mathcal{E}(W)$ (see Lemma \ref{lemma_on_EW_and_estimate_of_error}). Since $\vec{m}$ contains $\hatsig_s$ while the space variable $y$ is defined using $\sigma$, we introduce the quantity $\vec{n}$ whose second component will compensate a term appearing in the equation of $\vare$ (see Lemma \ref{Lemma_on_equation_of_vare}).

\end{remark}
 
\vspace{0.3cm}
The following lemma states technical findings on the behaviour of the perturbing tail in the rescaled variables in the proximity of the blow up profile.
\begin{lemma}\label{lemma_with_all_estimates_on_r_F_etc}
    Under the assumptions \eqref{BS1} and assuming
    \begin{equation}\label{assumption_on_sigma_on_t}
    \sigma(s)> \frac{4\times110}{3}\tau(s) + 3\,x_0 \quad \text{for all } \,\, s\in [s_0,+\infty),
    \end{equation}
    with $x_0 > 1$ and for $s_0$ large enough, the following estimates hold for $s \geq s_0 $
    \begin{equation}\label{estimates_on_r}
        \big| r - c_0 \lambda^{\frac{1}{2}} e^{-\frac 12 \hatsig} \big| \lesssim \delta(x^{-1}_0)s^{-\frac 52},\qquad  |r|\lesssim s^{-1},
    \end{equation}
    \begin{equation}\label{estimate_on_r_s}
         \Big| r_s - c_0 \partial_s \big( \lambda^{\frac12} e^{-\frac 12 \hatsig}\big)\Big| \lesssim \big( s^{-\frac52}|\vec{m}| + s^{-\frac 72} \big),
    \end{equation}
    \begin{equation}
        |r_s| \lesssim s^{-1}|\vec{m}| + s^{-2},
    \end{equation}
    \begin{equation}\label{estimate_on_r_minus_F}
        \big| r(s)-F(s,y) \big| \lesssim (1+y^2)s^{-2}, \qquad \forall y \in \RR,
    \end{equation}
    \begin{equation}\label{estimates_on_norms_L2_of_F}
        \|F(s)\|_{L^2_y} \lesssim e^{-\frac 18 x_0},
        \qquad \|\partial_y F(s) \|_{L^2_y} \lesssim e^{-\frac 18 x_0}\,s^{-1},
    \end{equation}
    \begin{equation}\label{estimates_on_norms_L_infty_of_F}
        \|F(s)\|_{L^{\infty}(y>-2b^{-\gamma})} \lesssim s^{-1}, \qquad \|\partial_y F\|_{L^{\infty}(y>-2b^{-\gamma})} \lesssim s^{-2},
    \end{equation}
    \begin{equation}\label{estimate_on_L_infty_of_dj_F}
        \| \partial^j_y \big( F(s,y) \big) \|_{L^{\infty}_y} \lesssim \delta(x^{-1}_0)s^{-j-\frac 12}, \qquad \text{for } j\geq 0,
    \end{equation}
    \begin{equation}\label{estimates_on_norms_L_infty_of_F_with_exp}
        | F(s,y)|  \lesssim (1+y^2)s^{-1}, \qquad | \partial_y F(s,y)| \lesssim (1+y^2) s^{-2},
    \end{equation}
    \begin{equation}\label{estimates_on_scalar_prod_r_minus_F_and_Q}
        \Big| \big( \partial_y(5 Q^4(F-r)), Q \big) + \frac 12 c_0 \big(\int Q \big)\lambda^{\frac 32}e^{-\frac 12 \hatsig}  \Big| \lesssim s^{-3}.
    \end{equation}
\end{lemma}
\begin{proof}
By \eqref{lemma_decaying_tails_estim_deriv_on_x_order_k}, the assumption \eqref{assumption_on_sigma_on_t} on $\sigma$ and \eqref{BS1}, we get
\begin{equation}
\begin{split}
    &\big|r - c_0\lambda^{\frac 12} e^{-\frac 12 \hatsig} \big| = \lambda^{\frac 12}\big| f(\tau(s), \sigma) - f_0(\tau(s), \sigma) \big| \lesssim \delta(x^{-1}_0)\lambda^{\frac 12}e^{-2\sigma} \\
    &\lesssim \delta(x^{-1}_0)e^{-\frac{\tau(s)}{2}}\lambda^{\frac 12}e^{-2\hatsig} \lesssim \delta(x^{-1}_0) s^{-\frac 52}.
\end{split}
\end{equation}
    
Therefore,
\begin{equation}
    |r| \lesssim \delta(x^{-1}_0) s^{-\frac 52} + \lambda^{\frac 12}e^{-\frac 12 \hatsig} \lesssim s^{-1}.
\end{equation}

For the second inequality, we write
\begin{equation}
\begin{split}
    &\big| r_s - \partial_s \big(c_0 \lambda^{\frac 12} e^{-\frac12 \hatsig} \big) \big|\\
    = &\big| \frac 12 \lambda_s \lambda^{-\frac 12} f(\tau(s),\sigma) + \lambda^{\frac 72}(\partial_1 f)(\tau(s), \sigma) + \lambda^{\frac 12}\sigma_s (\partial_x f)(\tau(s), \sigma) - \frac{c_0}{2}\lambda^{-\frac 12}\lambda_s e^{-\frac 12 \hatsig} + \frac{c_0}{2}\lambda^{\frac 12}\hatsig_s\, e^{-\frac 12 \hatsig} \big|\\
    \leq &\big| \frac{1}{2}\lambda^{\frac 12}\frac{\lambda_s}{\lambda}\big|\, \big| f(\tau(s), \sigma) - f_0(\tau(s), \sigma) \big| + \lambda^{\frac 72}\big| (\partial_1 f)(\tau(s), \sigma)\big| + \lambda^{\frac 12}|\sigma_s|\, \big| (\partial_x f)(\tau(s), \sigma) - (\partial_x f_0)(\tau(s), \sigma) \big|.
\end{split}
\end{equation}

Using $|\sigma_s| \lesssim \lambda|\vec{m}| + \lambda$, the Lemma \ref{lemma_on_the_decaying_tail} with \eqref{assumption_on_sigma_on_t} and \eqref{BS1}, we have
\begin{equation}
\begin{split}
    &\big| r_s - \partial_s \big(c_0 \lambda^{\frac 12} e^{-\frac12 \hatsig} \big) \big| \lesssim \lambda^{\frac 12}e^{-2\sigma}\Big(\Big| \frac{\lambda_s}{\lambda} +b\Big| + b\Big) + \lambda^{\frac 72}e^{-\frac 12 \sigma} + \lambda^{\frac 32}e^{-2\sigma}|\vec{m}| + \lambda^{\frac 32}e^{-2\sigma}\\ 
    &\lesssim e^{-\frac{\tau(s)}{8}}\big( s^{-\frac 52}|\vec{m}| + s^{-\frac 72}\big).
\end{split}
\end{equation}

Thus, we get the estimate on $|r_s|$
\begin{equation}
\begin{split}
    &|r_s| \lesssim  s^{-\frac 52}|\vec{m}| + s^{-\frac 72} + \big| \partial_s (\lambda^{\frac 12} e^{-\frac 12 \hatsig})\big|\\
    &\lesssim  s^{-\frac 52}|\vec{m}| + s^{-\frac 72} + \lambda^{\frac 12}\big(|\vec{m}| +b \big)e^{-\frac 12 \hatsig}+\lambda^{\frac 32}e^{-\frac 12 \hatsig}
    \lesssim s^{-1}|\vec{m}| + s^{-2}.
\end{split}
\end{equation}

In order to get the estimate \eqref{estimate_on_r_minus_F} on $\RR$, we separate the two cases, $\lambda |y| <1$ and $\lambda |y| >1$.

The following estimate holds for all $y \in \RR$
\begin{equation}\label{proof_technic_estim_on_r_F_1}
    |r -F| \lesssim \lambda^{\frac 32}|y|\sup_{x \in J}\big| (\partial_x f)\big|,
\end{equation}
where $J = [\min(\sigma, \lambda y +\sigma), \max(\sigma, \lambda y + \sigma)]$. 
\begin{itemize}
    \item For $\lambda|y| < 1$, we get that $\lambda y + \sigma > \sigma - 1$. Thus, $J \subset [\sigma-1, +\infty)$. The assumption \eqref{assumption_on_sigma_on_t} and $x_0 > 1$ gives that $\sigma $ and $\sigma-1$ both satisfy the condition \eqref{dec_t_condition_on_x} of Lemma \ref{lemma_on_the_decaying_tail}. By \eqref{proof_technic_estim_on_r_F_1}, we get for $\lambda|y| < 1$
    \begin{equation}
        \begin{split}
            |F-r| \lesssim \lambda^{\frac 32}|y|\sup_{y >\sigma-1}\big| (\partial_x f)\big| \lesssim |y|\lambda^{\frac 32}e^{-\frac 12 \hatsig} \lesssim |y|s^{-2}.
        \end{split}
    \end{equation}
    \item For $\lambda|y|\geq 1$, it holds
    \begin{equation}
        |F-r | \lesssim |y|\lambda^{\frac 32}\|\partial_x f\|_{L^{\infty}_x}\lesssim \delta(x^{-1}_0)y^2\lambda^{\frac 52} \lesssim y^2 s^{-\frac 52}.
    \end{equation}
\end{itemize}
Combining the estimates on two regions yields the estimate \eqref{estimate_on_r_minus_F}.

Since $\|f\|_{H^1_x}\lesssim \delta(x^{-1}_0)$, we get 
\begin{equation}
    \|F(s)\|_{L^2_y} = \|f(\tau(s))\|_{L^2_x} \lesssim e^{-\frac 18 x_0} \quad\text{and}\quad\|\partial_y F(s)\|_{L^2_y} = \lambda \|(\partial_x f)\|_{L^2_x}\lesssim s^{-1}e^{-\frac 18 x_0}.
\end{equation}
Thus, we have the estimate in \eqref{estimates_on_norms_L2_of_F}.

In order to get \eqref{estimates_on_norms_L_infty_of_F}, we remark that by \eqref{BS1} and $0<\gamma<1$, for $s_0 \gg 1$, it holds 
\begin{equation}
    \{ y > -2b^{-\gamma} \} \subset \{ y > -\lambda^{-1} \} = \{\lambda y +\sigma > \sigma-1\}.
\end{equation}
By \eqref{assumption_on_sigma_on_t}, $\sigma -1$ verifies the assumption on $x$ in the Lemma \ref{lemma_on_the_decaying_tail}. Hence, for $y > -2b^{-\gamma}$, we have
\begin{equation}
    |F(s,y)| \lesssim \lambda^{\frac 12}\big( e^{-2(\lambda y +\sigma)} + e^{-\frac 12 (\lambda y +\hatsig)} \big) \lesssim s^{-1}.
\end{equation}
Similarly, we get
\begin{equation}
    |\partial_y F(s,y)| \lesssim \lambda^{\frac 32}\big( e^{-2(\lambda y +\sigma)} + e^{-\frac{1}{2}(\lambda y +\hatsig)} \big) \lesssim s^{-2}.
\end{equation}

For the estimate \eqref{estimate_on_L_infty_of_dj_F}, we write
\begin{equation}
    \|\partial^j_y F\|_{L^{\infty}_y} \lesssim \lambda^{j + \frac 12}\|\partial^j_x f\|_{L^{\infty}_t L^{\infty}_x} \lesssim \delta(x^{-1}_0)\,s^{-j-\frac 12}.
\end{equation}

The estimates \eqref{estimates_on_r} and \eqref{estimate_on_r_minus_F} give the first inequality in \eqref{estimates_on_norms_L_infty_of_F_with_exp}
\begin{equation}
    |F| \lesssim |r| +(1+y^2)s^{-2} \lesssim (1+y^2)s^{-1}.
\end{equation}

In order to get the second estimate in \eqref{estimates_on_norms_L_infty_of_F_with_exp}, we introduce
\begin{equation}
    \lambda^{\frac 32}\big|(\partial_x f)(\tau(s), \sigma) - (\partial_x f)(\tau(s),\lambda y + \sigma)\big|.
\end{equation}
The same bound holds for $y\in \RR$
\begin{equation}
    \big|(\partial_x f)(\tau(s), \sigma) - (\partial_x f)(\tau(s),\lambda y + \sigma)\big| \lesssim \lambda |y| \sup_{J}|(\partial^2_x f)|,
\end{equation}
with $J = [\min(\sigma, \lambda y +\sigma), \max(\sigma, \lambda y + \sigma)]$. 
Proceeding as in the proof of \eqref{estimate_on_r_minus_F}, separating the study on two regions $\lambda |y| < 1$ and $\lambda |y| \geq 1$ and by the result on decaying tails and the exponential decay on $s$, we get 
\begin{equation}\label{estimate_F_0_and_F_in_the_proof_tech_estim}
    \lambda^{\frac 32}\big|(\partial_x f)(\tau(s), \sigma) - (\partial_x f)(\tau(s),\lambda y + \sigma)\big| \lesssim (1+y^2)s^{-3}.
\end{equation}
Therefore
\begin{equation}
    \big| \partial_y F\big| \lesssim (1+y^2)s^{-3} + \lambda^{\frac 32}\big| (\partial_x f)(\tau(s),\sigma)\big| \lesssim (1+y^2)s^{-2}.
\end{equation}

We now prove the estimate \eqref{estimates_on_scalar_prod_r_minus_F_and_Q}. 
By integration by parts, we have
\begin{equation}
    \big( \partial_y(5 Q^4(F-r)), Q \big) = \lambda^{\frac 32}\int Q^5 (\partial_x f)(\tau(s),\lambda y + \sigma).
\end{equation}
Since $\int Q^5 =\int Q$, we get
\begin{equation}
\begin{split}
    &\Big| \big( \partial_y(5 Q^4(F-r)), Q \big) + \frac 12 c_0 \big(\int Q \big)\lambda^{\frac 32}e^{-\frac 12 \hatsig}  \Big|
    = \Big| \big( \partial_y(5 Q^4(F-r)), Q \big) - \lambda^{\frac 32}\big(\int Q \big) (\partial_x f_0)(\tau(s),\sigma)  \Big|\\
    &\leq \lambda^{\frac 32}\int Q^5 \big| (\partial_x f)(\tau(s), \lambda y +\sigma) -(\partial_x f_0)(\tau(s), \sigma) \big|\\
    &\leq \lambda^{\frac 32}\int Q^5 \big| (\partial_x f)(\tau(s), \lambda y +\sigma) -(\partial_x f)(\tau(s), \sigma) \big| + \lambda^{\frac 32}\int Q^5 \big| (\partial_x f)(\tau(s), \sigma) -(\partial_x f_0)(\tau(s), \sigma) \big|.
\end{split}
\end{equation}
The first term on the r.h.s. is estimated using the exponential decay of $Q$ and \eqref{estimate_F_0_and_F_in_the_proof_tech_estim}.

Using Lemma \ref{lemma_on_the_decaying_tail}  and \eqref{assumption_on_sigma_on_t}, for the second term on the r.h.s., we get
\begin{equation}
    \lambda^{\frac 32}\int Q^5 \big| (\partial_x f)(\tau(s), \sigma) -(\partial_x f_0)(\tau(s), \sigma) \big| \lesssim \lambda^{\frac 32}e^{-2\hatsig}\, \delta(x^{-1}_0)\int Q^5 \lesssim s^{-\frac 72}
\end{equation}
Combining the two estimates yields \eqref{estimates_on_scalar_prod_r_minus_F_and_Q}.
    
\end{proof}
Recall that $c_1 = \frac{1}{16}\Big(\int Q\Big)^2$.
\begin{lemma}[Mass and energy of $W+F$]
Under the assumption \eqref{BS1}, for $s_0$ large enough, it holds for $s\geq s_0$
\begin{equation}\label{estimate_on_mass_W_plus_F}
    \bigg\lvert \int (W+F)^2-\Big( \int Q^2 + 2b\int P Q + \frac{1}{2}r \int Q \Big) \bigg\rvert \lesssim s^{-2+\gamma} + e^{-\frac 14 x_0},
\end{equation}
\begin{equation}\label{estimate_on_energy_W_plus_F}
 \big\lvert   E(W+F) + c_1\,\lambda^2\,g(s) \big\rvert \lesssim  s^{-2}.
\end{equation}
\end{lemma}
\begin{proof}
Using \eqref{properties_of_R}, we write
\begin{equation}
    \int (W+F)^2 - \int Q^2_b - \frac 12 r \int Q = 2br \int P_b R + 2 \int Q(F-r) + 2 b \int P_b F + r^2 \int R^2 + 2r \int RF + \int F^2.
\end{equation}
Therefore, by the estimates in Lemma \ref{lemma_with_all_estimates_on_r_F_etc}, the conservation of the mass $\|F\|_{L^2_y} = \|f_0\|_{L^2_x}$ and \eqref{estimation_of_norms_of_f_0}, we have
\begin{equation}
\begin{split}
    \Big| \int (W+F)^2 - Q^2_b - \frac 12 r Q \Big| \lesssim s^{-2} + s^{-1}\|F\|_{L^{\infty}(y > -2b^{-\gamma})}\int |P_b| + \|F\|^2_{L^2_y}\lesssim s^{-2+\gamma} + e^{-\frac14 x_0}.
\end{split}
\end{equation}
Combining the previous estimate with \eqref{estimate_mass_of_Qb} yields \eqref{estimate_on_mass_W_plus_F}.

The energy of $W+F$ is rewritten as follows
\begin{equation}
    \begin{split}\label{proof_energy_W_plus_F_rewriting_of_the_energy}
        E(W+F) &= E(Q_b) + E(F) + \int Q'_b \partial_y(rR+F)  + r\int R'\partial_y F + \frac 12 r^2 \int (R')^2\\
        &- \int Q^5_b(rR+F)  - \frac 16 \int \big[(Q_b + rR+F)^6 - Q^6_b - F^6 - 6 Q^5_b (rR+F) \big].
    \end{split}
\end{equation}

By Lemma \ref{lemma_with_all_estimates_on_r_F_etc} and the definition of $P_b$, we have
\begin{equation}
    \big|\int \big(Q'_b  -  Q'\big) \partial_y (rR+F)\big| \lesssim b\int |P'_b | \big(|rR'|+|\partial_y F|\big) \lesssim s^{-2} + s^{-1}\|\partial_y F\|_{L^{\infty}(y>-2b^{-\gamma})} \lesssim s^{-2}.
\end{equation}
and 
\begin{equation}
    \big| \int \big( Q^5 - Q^5_b\big) (rR+F)\big|\lesssim b^5 \int |P_b|^5 \big(|rR|+|F|\big) + b\int |P_b |Q^4 \big(|rR|+|F|\big) \lesssim s^{-2}.
\end{equation}
The estimates \eqref{estimates_on_r}, \eqref{estimates_on_norms_L_infty_of_F_with_exp}, the definition of $\chi_b$ and the decay of $R$, yield
\begin{equation}
    \begin{split}
        \big| \int \big[(Q_b + rR+F)^6 - Q^6_b - F^6 - 6 Q^5_b (rR+F) \big]\big| \lesssim \int Q^4_b (rR+F)^2 +\int \big|(rR+F)^6 - F^6\big| \lesssim s^{-2}
    \end{split}
\end{equation}
and
\begin{equation}
    \big|r \int R'\,\partial_y F\big| + r^2 \int (R')^2  \lesssim s^{-2}.
\end{equation}

Using $E(W+F)$ in \eqref{proof_energy_W_plus_F_rewriting_of_the_energy}, the equation of $Q$ and the above estimates, we obtain
\begin{equation}
\begin{split}
    &\big| E(W+F) + c_1 \, \lambda^2\,g(s) \big| \lesssim  \Big| E(Q_b) - \int Q(rR+F) + c_1 \big( b + \frac{4}{\int Q} c_0 \lambda^{\frac 12}e^{-\frac 12 \hatsig}\big) \Big| + |E(F)| + s^{-2}\\
    &\lesssim \big|E(Q_b) + b(P,Q) \big| + \frac{1}{4}\int Q \big|r - c_0 \lambda^{\frac 12} e^{-\frac 12 \hatsig}\big| + \int Q |F-r|+ |E(F)| + s^{-2} \lesssim s^{-2}.
\end{split}
\end{equation}
For the last inequality, we have used the equation of $P$ and of $R$, the estimate on $E(Q_b)$ in \eqref{estimate_energy_of_Qb} and Lemma \ref{lemma_with_all_estimates_on_r_F_etc}.

We have by \eqref{BS1}, \eqref{estimation_of_norms_of_f_0} and conservation of the energy
\begin{equation}
    E(F(s)) = \lambda^2 E(f(\tau(s))) = \lambda^2 E(f_0) \lesssim s^{-2} e^{-\frac 14 x_0}.
\end{equation}

\end{proof}

\begin{lemma}\label{lemma_on_EW_and_estimate_of_error}
    For $W$ defined in \eqref{def_of_W_and_r}, the error term $\mathcal{E}(W)$ defined in \eqref{def_of_error_term_mathcal_E} can be rewritten as
    \begin{equation}\label{equation_rewriting_of_error}
        \mathcal{E}(W) = -\vec{m}\cdot\vec{M}Q + \mathcal{R},
    \end{equation}
    where for $s_0$ large enough, for $s \geq s_0$, it holds
    \begin{equation}\label{estimate_on_norms_L_2_B_of_R}
        \|\mathcal{R} \|_{L^2_B} + \| \partial_y \mathcal{R}\|_{L^2_B} \lesssim |b_s| + s^{-1}|\vec{m}|+s^{-2}
    \end{equation}
    and
    \begin{equation}\label{estimate_on_scalar_prod_R_Q}
        \bigg\lvert\big( \mathcal{R}, Q\big) -c_1\,\bigg[  b_s+2b^2 +\lambda^2 \partial_s \Big( \frac{4}{\int Q} c_0 \lambda^{-\frac 32} e^{-\frac 12 \hatsig}\Big) \bigg] \bigg\rvert \lesssim  |b_s|s^{-100} + s^{-3} + s^{-1}|\vec{m}|.
    \end{equation}
\end{lemma}
\begin{proof}
The equation of $R$ and the definition of $W, Q_b, \Psi_b$ and of $\vec{m}, \vec{M}$ yields
\begin{equation}
    \mathcal{E}(W) = - \vec{m}\cdot\vec{M}\, Q+\mathcal{R},
\end{equation}
with 
\begin{equation}
    \begin{split}
        &\mathcal{R} = \mathcal{R}_1 + \partial_y \mathcal{R}_2 + \mathcal{R}_3 - \Psi_b,\\
        &\mathcal{R}_1 = r_s \,R + b\,r\, \Lambda R + 5 \partial_y \big[ Q^4_b(rR+F)-Q^4(rR+r) \big],\\
        & \mathcal{R}_2 = \big( Q_b + rR + F \big)^5 - Q^5_b - F^5 - 5 Q^4_b (rR+F),\\
        &\mathcal{R}_3 = (b_s+2b^2)\,\partial_b Q_b - b\, \vec{m}\cdot \vec{M} P_b - r \,\vec{m}\cdot \vec{M}R.
    \end{split}
\end{equation}

The projection of $\mathcal{R}$ on $Q$ is rewritten as follows
\begin{equation}
\begin{split}
    &\big( \mathcal{R}, Q\big) -c_1\,\bigg[  b_s+2b^2 +\lambda^2 \partial_s \Big( \frac{4}{\int Q} c_0 \lambda^{-\frac 32} e^{-\frac 12 \hatsig}\Big) \bigg]
    =  \big(\mathcal{R}_1, Q \big) - c_1 \,\lambda^2 \, \partial_s \Big(\frac{4}{\int Q}\,c_0\,\lambda^{-\frac 32}e^{-\frac 12 \hatsig} \Big)\\
    &+ \big(\partial_y \mathcal{R}_2, Q \big)
    + \big(\mathcal{R}_3, Q \big) - c_1\, \big(b_s+2b^2\big) - \big(\Psi_b, Q\big).
\end{split}
\end{equation}

We remark using the definition of $c_1$ that
\begin{equation}
    - c_1 \, \lambda^2\,\partial_s\Big(\frac{4}{\int Q}c_0 \lambda^{-\frac 32}e^{-\frac 12 \hatsig} \Big) = \frac 34 \Big( \int Q \Big) \,\partial_s \Big( c_0\lambda^{\frac 12 } e^{-\frac 12 \hatsig}\Big) + \frac{c_0}{2}\Big(\int Q \Big)\lambda^{\frac 12}\, \hatsig_s\, e^{-\frac 12 \hatsig}
\end{equation}
and the following rewriting holds true for $\phi \in \mathcal{Y}$
\begin{equation}\label{proof_estim_of_mathcal_R_rewriting_in_Y}
    5 \big( \partial_y \big[ Q^4_b(rR+ F) - Q^4(rR+r) \big], \phi \big) = -5\big( (Q^4_b - Q^4)(rR+F), \phi'\big) - 5 \big( Q^4(F-r), \phi' \big).
\end{equation}

Therefore, using the above relation for $\phi = Q$, we get
\begin{equation}
    \begin{split}
        &\bigg| \big( \mathcal{R}_1, Q \big) - c_1\, \lambda^2\partial_s \Big(\frac{4}{\int Q}c_0 \lambda^{-\frac32 } e^{-\frac 12 \hatsig} \Big) \bigg|
        \leq \big|r_s \big(\mathcal{R}, Q \big) + \frac32 \Big( \int Q\Big) \partial_s \Big(c_0 \lambda^{\frac 12}e^{-\frac 12 \hatsig} \Big) \big|\\
        &+ \big| br \big( \Lambda R , Q\big) - 20 br \big( Q^3(R+1)P_b, Q' \big) \big|
        + \big| 5\big( (Q^4_b - Q^4)(rR+F), Q'\big) - 20br \big( Q^3(R+1)P_b,Q' \big) \big|\\
        &+ \big|\big( 5Q^4(F-r),Q' \big) - \frac{c_0}{2}\Big( \int Q\Big) \lambda^{\frac 12}\,\hatsig_s\, e^{-\frac 12 \hatsig}\big|
    \end{split}
\end{equation}
Since $(R, Q) = -\frac 34 \int Q$, by \eqref{estimate_on_r_s} we have 
\begin{equation}
    \big|r_s \big(\mathcal{R}, Q \big) + \frac34 \Big( \int Q\Big) \partial_s \Big(c_0 \lambda^{\frac 12}e^{-\frac 12 \hatsig} \Big) \big| \lesssim \big|r_s - \partial_s\big( c_0 \lambda^{\frac 12} e^{-\frac 12 \hatsig}\big)\big| \lesssim e^{-\frac 18 \tau(s)}\big( s^{-\frac 52} |\vec{m}| + s^{-\frac 72}\big).
\end{equation}

Since $\mathcal{L}R = 5Q^4$ and using the following relation
\begin{equation}
    \Lambda Q + 2PQ^3Q' = (\mathcal{L} P)' + 20 Q^3 Q'P = \mathcal{L}(P'),
\end{equation}
it holds
\begin{equation}
    br(\Lambda R, Q) = - br(R,Q) = -br(R, \mathcal{L}(P')) + 20br(R,PQ^3Q') = 20br(Q^3(R+1)P,Q').
\end{equation}
Therefore, $\supp(1-\chi_b)$ imply
\begin{equation}
    \big| br \big( \Lambda R , Q\big) - 20 br \big( Q^3(R+1)P_b, Q' \big) \big| \lesssim |br|\big|\big( Q^3(R+1)Q', P-P_b \big)\big| \lesssim s^{-2}\int^{-b^{-\gamma}}_{-\infty} e^{-|y|} \lesssim s^{-100}.
\end{equation}

Using the rewriting $rR+F = r(R+1) -(F-r)$, we get
\begin{equation}
   \big| 5 \big( (Q^4_b -Q^4)(rR+F), Q' \big) - 20br\big(Q^3(R+1)P_b, Q'\big) \big| \lesssim s^{-3}.
\end{equation}

By \eqref{estimates_on_scalar_prod_r_minus_F_and_Q} and $|\hatsig_s - \lambda| \lesssim \lambda|\vec{m}|$, we get
\begin{equation}
    \big|\big( 5Q^4(F-r),Q' \big) - \frac{c_0}{2}\Big( \int Q\Big) \lambda^{\frac 12}\,\hatsig_s\, e^{-\frac 12 \hatsig}\big| \lesssim s^{-3} + \lambda^{\frac 12 }e^{-\frac 12 \hatsig}|\hatsig_s - \lambda| \lesssim s^{-3} + s^{-2}|\vec{m}|.
\end{equation}
Combining the above estimates, it yields
\begin{equation}\label{proof_estimate_proj_R1_on_Q}
    \Big| \big(\mathcal{R}_1, Q \big) - c_1 \,\lambda^2 \, \partial_s \Big(\frac{4}{\int Q}\,c_0\,\lambda^{-\frac 32}e^{-\frac 12 \hatsig} \Big)\Big| \lesssim  s^{-3} + s^{-2}|\vec{m}|.
\end{equation}

In order to estimate the projection of $\mathcal{R}_2$ on $Q'$, we use the following rewriting
\begin{equation}
    \mathcal{R}_2 = 10Q^3(rR+F)^2 + 10(Q^3_b -Q^3)(rR+F)^2 + 10Q^2_b(rR+F)^3 + 5Q_b(rR+F)^4 + (rR+F)^5 - F^5.
\end{equation}
The first term in this rewriting is developed as follows
\begin{equation}
    \big( 10Q^3(rR+F)^2, Q' \big) = -5r^2\big( (R+1)\,R', Q^4\big) - 5r \big( (F-r)\,R', Q^4 \big) - 5r\big( R\,\partial_y F, Q^4\big) - 5 \big( F\,\partial_y F, Q^4 \big).
\end{equation}
The equation of $R$ in \eqref{properties_of_R} yields
\begin{equation}
     5\big( (R+1)\,R', Q^4\big) = \big( R-R'', R' \big) =0 .
\end{equation}
Combining the above rewritings with \eqref{estimate_on_r_minus_F}, \eqref{estimates_on_norms_L_infty_of_F_with_exp} and \eqref{estimates_on_r}, we have
\begin{equation}\label{proof_estimate_proj_dR2_on_Q}
    \big|\big( \mathcal{R}_2, Q' \big)\big| \leq \big|\big( 10Q^3(rR+F)^2, Q'\big)\big| + Cs^{-3} \lesssim s^{-3}.
\end{equation}

For the estimate of $\mathcal{R}_3$ in the direction of $Q$,
by $(P,Q)= c_1 $ and the definition of $\chi_b$, we have
\begin{equation}
    |b_s + b^2| \big| \big( \partial_b Q_b, Q\big) - c_1 \big| \lesssim |b_s + b^2| \big|(P,Q) - \big((\chi_b -1)P, Q\big) + \gamma \big( y P \chi'_b, Q\big) - c_1\big| \lesssim |b_s+b^2|e^{-cs^{\gamma}}.
\end{equation}
Therefore
\begin{equation}\label{proof_estimate_proj_R3_on_Q}
\begin{split}
    &\big|\big( \mathcal{R}_3, Q \big) - c_1 (b_s + 2b^2) \big| \lesssim |b_s + b^2| \big| \big( \partial_b Q_b, Q\big) - c_1 \big| + s^{-1}|\vec{m}| \lesssim |b_s|s^{-100} + s^{-100} + s^{-1}|\vec{m}|.
\end{split}
\end{equation}
Combining \eqref{proof_estimate_proj_R1_on_Q}, \eqref{proof_estimate_proj_dR2_on_Q}, \eqref{proof_estimate_proj_R3_on_Q} and the projection of $\Psi_b$ on $Q$ in \eqref{projection_of_Psi_b_on_Q}, we get \eqref{estimate_on_scalar_prod_R_Q}.

The definition of $\chi$ in \eqref{def_of_chi}, the weighted norm of $\Psi_b$ in \eqref{projection_Psi_b_on_Y_and_norm_L2B} and the decay properties of $P$ on the right and of $Q_b$ yields
\begin{equation}
    \|\mathcal{R}_1\|_{L^2_B} + \|\mathcal{R}_2\|_{L^2_B} +  \|\partial_y\mathcal{R}_1\|_{L^2_B}+ \|\partial_y \mathcal{R}_2\|_{L^2_B}\lesssim s^{-2}
\end{equation}
and 
\begin{equation}
    \|\mathcal{R}_3\|_{L^2_B} + \|\partial_y \mathcal{R}_3\|_{L^2_B} \lesssim |b_s|+ s^{-2} + s^{-1}|\vec{m}|.
\end{equation}
Thus, we get \eqref{estimate_on_norms_L_2_B_of_R}.

\end{proof}

\begin{lemma}[Estimates on $h_s$ and $g_s$]\label{lemma_with_estimates_on_hs_and_gs}
For $s \geq s_0 $, it holds
\begin{equation}\label{estimate_on_gs_with_prod_scal_R_Q}
    \big| \big(\mathcal{R},Q\big)- c_1 \lambda^2 \partial_s g \big| \lesssim |b_s|s^{-100}+s^{-3} + s^{-1}|\vec{m}|,
\end{equation}
\begin{equation}\label{estimate_on_hs_with_g}
    \big\lvert \lambda^{-\frac{1}{2}}\partial_s h + \frac{1}{2}\lambda^2 g \big\rvert \lesssim |\vec{m}|.
\end{equation}
\end{lemma}
\begin{proof}

Combining 
\begin{equation}\label{rewriting_lambda_2_g_s}
\lambda^2 \partial_s g = b_s + 2b^2 + \lambda^2 \partial_s\Big( \frac{4}{\int Q}c_0 \lambda^{-\frac 32}e^{-\frac 12 \hatsig} \Big) - 2b\Big( \frac{\lambda_s}{\lambda}+b \Big).
\end{equation}
with \eqref{estimate_on_scalar_prod_R_Q} we get the estimate \eqref{estimate_on_gs_with_prod_scal_R_Q}.

Moreover,
\begin{equation}
    \lambda^{-\frac 12} \partial_s h +\frac 12 \lambda^2 g = \frac 12\Big( \frac{\lambda_s}{\lambda} + b \Big) - \frac{2}{\int Q}c_0 \,\lambda^{\frac 12} e^{-\frac 12 \hatsig} \Big(\frac{\hatsig_s}{\lambda} - 1 \Big).
\end{equation}
combined with \eqref{BS1} gives \eqref{estimate_on_hs_with_g}.

\end{proof}

\subsection{Modulation and parameters estimates}

Let $U $ be a solution of \eqref{gKdV_principal_eq} defined on time interval $I$ included in $[0,T)$. Consider $s_0 >1$ to be chosen.
For the function $f$ introduced in \eqref{equation_of_f}, define
\begin{equation}
    v(t,x) = U(t,x) - f(t,x),
\end{equation}
which is said to be \textit{close to the approximate blow up profile}, if for all $t \in I$ there exist $(\lambda_{\sharp}(t), \sigma_{\sharp}(t))\in (0,+\infty)^2$, a universal small enough constant $\alpha^*>0$ and $z \in C(I,L^2(\RR))$ such that
\begin{equation}\label{supposed_decomposition_around_Q}
\begin{split}
    v(t,x)=\lambda^{-\frac{1}{2}}_{\sharp}(t)(Q+r_{\sharp}R+z)(t,y),\quad
     y= \frac{x-\sigma_{\sharp}(t)}{\lambda_{\sharp}(t)},\quad  r_{\sharp}(t)= \lambda^{\frac{1}{2}}_{\sharp}(t)f(t,\sigma_{\sharp}(t)).
\end{split}
\end{equation}
and for all $t \in I$, it holds
\begin{equation}\label{supposed_decomposition_around_Q_condition}
    \|z(t)\|_{H^1}+ \sigma_{\sharp}^{-1}(t)+\lambda_{\sharp}(t ) \leq \alpha^{*}.\\
\end{equation}


\begin{lemma}\label{lemma_on_decomposition_around_Q}
    Assume $v$ is close to the approximate blow up profile for some $\alpha^{*}>0$ sufficiently small. There exist $s^*>s_0$ and unique continuous functions $(\lambda,\sigma,b):[s_0,s^*]\mapsto (0,+\infty) \times \RR^2$ such that
    \begin{equation}\label{decomposition_of_v_in_lemma}
        \lambda^{\frac{1}{2}}(s)v(\tau(s), \lambda(s)y+\sigma(s)) = W(s,y)+ \vare(s,y),
    \end{equation}
    where $W$ is defined in \eqref{def_of_W_and_r} and $\vare$ satisfies for all $s\in [s_0,s^*]$
    \begin{equation}\label{orthogonal_conditions_in_lemma}
        (\vare(s),y\Lambda Q ) = (\vare(s),\Lambda Q) = (\vare(s), Q) = 0.
    \end{equation}
    Moreover, for all $s \in [s_0,s^*]$ it holds
    \begin{equation}
        \bigg\lvert \frac{\lambda(s)}{\lambda_{\sharp}(\tau(s))} - 1\bigg\rvert + |b(s)| + \bigg\lvert \frac{\sigma(s)-\sigma_{\sharp}(\tau(s))}{\lambda_{\sharp}(\tau(s))} \bigg\rvert \lesssim \|z(\tau(s))\|_{L^2_{sol}}
    \end{equation}
    and
    \begin{equation}\label{estimates_on_norms_of_vares_in_z_in_decomp}
        \|\vare(s)\|_{L^2} \lesssim \|z(\tau(s))\|^{\frac 58}_{L^2}, \qquad \|\vare(s)\|_{L^2_{sol}} \lesssim \|z(\tau(s))\|_{L^2_{sol}},
    \end{equation}
    \begin{equation}\label{estimates_on_norms_of_partial_vares_in_z_in_decomp}
        \|\partial_y \vare(s)\|_{L^2} \lesssim \|z(\tau(s))\|_{L^2_{sol}} + \|\partial_y z(\tau(s))\|_{L^2},  \qquad \|\partial_y \vare(s)\|_{L^2_{sol}} \lesssim \|z(\tau(s))\|_{L^2_{sol}}+\|\partial_y z(\tau(s))\|_{L^2_{sol}}.
    \end{equation}
\end{lemma}  
\begin{proof}
Such a decomposition holds due to the orthogonality conditions on $\vare$ and the implicit function theorem. Control of the growth of parameters is achieved by projecting the expression of $\vare$ onto the orthogonal directions. This also yields control over the $H^1$ norm of $\vare$.
The proof is detailed in \cite[Section III.4]{1M}.
\end{proof}

\begin{lemma}[Equation of $\vare_s$]\label{Lemma_on_equation_of_vare}
The map $s\mapsto (\lambda(s),\sigma(s), b(s))$ is  $C^1$ and $\vare$ verifies the equation
    \begin{equation}\label{equation_of_eps_s}
        \partial_s \vare = \partial_y \Big[ -\partial^2_y\vare + \vare - \Big( (W+F+\vare)^5 - (W+F)^5 \Big) \Big] - \mathcal{E}(W)
        + \vec{m}\cdot\vec{M}\vare + \vec{n}\cdot\vec{M}\vare,
    \end{equation}
    where $\mathcal{E}(W)$ is defined in \eqref{def_of_error_term_mathcal_E} and the vectors $\vec{M}, \vec{m}, \vec{n}$ are defined in \eqref{def_of_vec_M_m_n}.
\end{lemma}
\begin{proof}
The equation of $\vare$ is derived using standard calculations. The regularity of the parameters is derived from Cauchy–Lipschitz theory. For a similar argument, see \cite[Lemma 2.7]{Combet-Martel-18}.
\end{proof}

\vspace{0.3cm}
The following lemma states the estimates on $\vare$ induced by the conservation laws.
\begin{lemma}[Mass and energy estimates]
    Under the bootstrap assumption \eqref{BS1}, for $s_0 \gg 1$, the following estimates hold on $[s_0,s^*]$
    \begin{equation}\label{estimate_on_L2_norm_of_vare}
        \|\vare(s) \|^2_{L^2} \lesssim \Big|\|U_0\|^2_{L^2_x} - \|Q\|^2_{L^2}\Big| + s^{-1} + e^{-\frac 14 x_0},
    \end{equation}
    \begin{equation}\label{estimate_on_L2_norm_of_grad_vare_before_BS2}
        \|\partial_y \vare(s)\|^2_{L^2} \lesssim s^{-2} + \lambda^2 \big|E(U_0)\big| + \lambda^2|g(s)| + \|\vare\|^2_{L^2_{sol}}.
    \end{equation}
\end{lemma}

\begin{proof}
Conservation of the mass gives
\begin{equation}
    \|U_0\|^2_{L^2_x} = \|U\|^2_{L^2_x} = \|W+F+\vare\|^2_{L^2_y} = \|W+F\|^2_{L^2_y}+\|\vare\|^2_{L^2} + 2 \int (W+F)\vare.
\end{equation}
Thus,
\begin{equation}
    \|\vare\|^2_{L^2}\leq \Big|\|U_0\|^2_{L^2_x} - \int Q^2 \Big| + \Big| \int (W+F)^2 - \int Q^2 \Big| + 2 \Big|\int (W+F)\vare \Big|.
\end{equation}
By \eqref{estimate_on_mass_W_plus_F}, we have
\begin{equation}
    \Big| \int (W+F)^2 - \int Q^2 \Big| \lesssim  s^{-1} + e^{-\frac 14 x_0}.
\end{equation}
The orthogonality condition $(\vare, Q) =0 $ and the Young inequality yields
\begin{equation}
    \begin{split}
        \Big| \int (W+F)\vare \Big| &\lesssim \big(b\|P_b\|_{L^2} + |r|\|R\|_{L^2} +\|F\|_{L^2}\big)\|\vare\|_{L^2}\\
        &\leq \frac{1}{2^{10}}\|\vare\|^2_{L^2} + Cs^{-2+\gamma} + Ce^{-\frac 14 x_0}.
    \end{split}
\end{equation}
Combining these estimates yields \eqref{estimate_on_L2_norm_of_vare}.

By energy conservation, it holds
\begin{equation}
\begin{split}
    \lambda^2 E(U_0) &= \lambda^2 E(U) = E(W+F+\vare)\\
    &=E(W+F) + \int \partial_y(W+F)\,\partial_y \vare + \frac12 \|\partial_y \vare\|^2_{L^2} - \frac 16 \int \big((W+F+\vare)^6 - (W+F)^6 \big).
\end{split}
\end{equation}
We rewrite the following term
\begin{equation}
    \int \partial_y(W+F)\,\partial_y \vare = -\int Q''\vare + b \int P'_b\,\partial_y \vare + r\int R' \,\partial_y \vare - \int \partial^2_y F\,\vare.
\end{equation}

Therefore
\begin{equation}
    \begin{split}
        \frac 12\|\partial_y \vare\|^2_{L^2} &\leq \lambda^2 \big|E(U_0)\big| + \Big| E(W+F) - \int (Q'' + Q^5)\vare + b\int P'_b\,\partial_y \vare + r\int R'\,\partial_y \vare\\
        &-\int (\partial^2_y F + F^5)\vare -\frac 16 \int \big((W+F+\vare)^6 -(W+F)^6 - 6Q^5\vare - 6 F^5 \vare\big)\Big|.
    \end{split}
\end{equation}
The term $\int \big(Q'' + Q^5\big) \vare  = 0$. 

By Young inequality, it holds
\begin{equation}
    \begin{split}
        \big| b\int P'_b\,\partial_y \vare + r\int R'\,\partial_y \vare \big| \lesssim  s^{-1}\big(\|P'_b\|_{L^2} + \|R'\|_{L^2} \big)\, \|\partial_y \vare\|_{L^2}
        \leq \frac{1}{2^{10}}\|\partial_y \vare\|^2_{L^2} + Cs^{-2}.
    \end{split}
\end{equation}

Holder inequality and Gagliardo-Nirenberg inequality \eqref{gagliardo-nirenberg} yield (using \eqref{estimates_on_norms_L2_of_F})
\begin{equation}
\begin{split}
    \big| \int (\partial^2_y F + F^5)\vare \big| &\lesssim \|\partial_y F\|_{L^2}\|\partial_y \vare\|_{L^2} + \int \vare^6 + \int F^6\\
    &\leq \frac{1}{2^{10}}\|\partial_y \vare\|^2_{L^2} + C \|\partial_y F\|^2_{L^2}\big( 1 + \|F\|^4_{L^2} \big) \leq \frac{1}{2^{10}}\|\partial_y \vare\|^2_{L^2} + C s^{-2} e^{-\frac 12 x_0}.
\end{split}
\end{equation}

Interpolating, using that $\int P_b \lesssim s^{\gamma } + 1$ with $ 0< \gamma < 1$ and the estimates \eqref{estimates_on_norms_L_infty_of_F}, \eqref{estimates_on_norms_L_infty_of_F_with_exp}, we get
\begin{equation}
\begin{split}
    &\int \big|(W+F+\vare)^6 - (W+F)^6 - 6Q^5 \vare - 6 F^5 \vare \big|\\
    &\lesssim \int |bP_b + rR |Q^4 |\vare| + \int |bP_b + rR|^5 |\vare| + \int W^4 |F \vare| + \int \vare^6\\
    &\lesssim s^{-2} + \|\vare\|^2_{L^2_{sol}}
    + \delta(\alpha^*)\|\partial_y \vare\|^2_{L^2}.
\end{split}
\end{equation}
Therefore, for $\alpha^*$ small enough, we have
\begin{equation}
    \|\partial_y \vare\|^2_{L^2} \lesssim s^{-2} + \lambda^2 \big|E(U_0)\big| + \big|E(W+F)\big| + \|\vare\|^2_{L^2_{sol}}.
\end{equation}
Combining with the estimate on $E(W+F)$ in \eqref{estimate_on_energy_W_plus_F} yields \eqref{estimate_on_L2_norm_of_grad_vare_before_BS2}.
\end{proof}

\begin{lemma}[Modulation estimates]\label{lemma_mod_estimates_after_decomp}
    Assume \eqref{BS1} and 
    \begin{equation}
    \sigma(s)> \frac{4\times110}{3}\tau(s) + 3\,x_0 \quad \text{for all } \,\, s\in [s_0,s^*].
    \end{equation}
    Then, for $s_0\gg1$, the following estimates hold on $[s_0,s^*]$
    \begin{equation}\label{estimate_on_m_mod_estim}
        |\vec{m}| \lesssim \| \vare \|_{L^2_{sol}} + s^{-2},
    \end{equation}
    \begin{equation}\label{estimate_on_bs_mod_estim}
        |b_s| \lesssim \|\vare\|^2_{L^2_{sol}} + s^{-2},
    \end{equation}
    \begin{equation}\label{estimate_on_gs_mod_estim}
        \lambda^2|\partial_s g| \lesssim s^{-3}+\|\vare \|^2_{L^2_{sol}} + s^{-1}\|\vare \|_{L^2_{sol}}.
    \end{equation}
\end{lemma}
\begin{proof}
  In order to get these estimates, we differentiate in time the three orthogonality conditions in \eqref{orthogonal_conditions_in_lemma}.
  
  Since $(Q', \Lambda Q) = (\Lambda Q, y \Lambda Q)= 0$ and $(Q', y\Lambda Q) = \|\Lambda Q\|^2_{L^2}$, we get from the orthogonality conditions on $\Lambda Q$ and $y \Lambda Q$, that
  \begin{equation}
      \begin{split}
           - \Big( \frac{\lambda_s}{\lambda } + b\Big)\|\Lambda Q\|^2_{L^2}
          &= -\big( \vare, \mathcal{L}(\Lambda Q)'\big) +\big( (W+F+\vare)^5
          - (W+F)^5 - 5Q^4\vare, (\Lambda Q)' \big)\\
          &- \big(\mathcal{R}, \Lambda Q \big) + \big( \vec{m}\cdot\vec{M}\vare, \Lambda Q \big) + \big( \vec{n}\cdot \vec{M} \vare, \Lambda Q\big)
      \end{split}
  \end{equation}
    and
\begin{equation}
    \begin{split}
         - \Big( \frac{\hatsig_s}{\lambda } - 1\Big)\|\Lambda Q\|^2_{L^2}
        &= -\big(\vare, \mathcal{L}(y\Lambda Q)'  \big) + \big( (W+F+\vare)^5
        - (W+F)^5 - 5Q^4\vare, (y\Lambda Q)'\big)\\
        &- \big(\mathcal{R}, y\Lambda Q \big) + \big( \vec{m}\cdot\vec{M}\vare, y \Lambda Q\big) + \big( \vec{n}\cdot \vec{M} \vare, y\Lambda Q\big).
    \end{split}
\end{equation}

Since $\Lambda Q$, $y\Lambda Q$ and their derivatives belong to $\mathcal{Y}$, we estimate the terms for any $\phi \in\mathcal{Y}$.
It holds,
\begin{equation}
    \big| \big( \vare, \phi \big) \big| \lesssim \|\vare\|^2_{L^2_{sol}}.
\end{equation}

By interpolation and Lemma \ref{lemma_with_all_estimates_on_r_F_etc}, we get
\begin{equation}
   \big| \big((W+F+\vare)^5 - (W+F)^5 - 5Q^4\vare, \phi \big) \big| \lesssim s^{-1}\|\vare\|_{L^2_{sol}} + \|\vare\|^2_{L^2_{sol}} + \int |\vare|^5 \phi \lesssim s^{-1}\|\vare\|_{L^2_{sol}} + \|\vare\|^2_{L^2_{sol}},
\end{equation}
where the nonlinear term is estimated with $\|\vare\|_{L^{\infty}}$ and the weighted norm $\|\vare\|^2_{L^2_{sol}}$.

For such $\phi \in \mathcal{Y}$ it also holds
\begin{equation}
    \big|\big( \vec{m}\cdot\vec{M}\vare, \phi \big)\big| \lesssim |\vec{m}|\,\|\vare\|_{L^2_{sol}}\qquad \text{and} \qquad \big|\big( \vec{n}\cdot\vec{M}\vare, \phi\big)\big| \lesssim s^{-1}\|\vare\|_{L^2_{sol}}.
\end{equation}

Using the above inequalities and \eqref{estimate_on_norms_L_2_B_of_R}, we get
\begin{equation}\label{proof_diff_in_s_of_ortho_first_estim_on_m}
    |\vec{m}|\lesssim \|\vare\|_{L^2_{sol}} + |b_s| + s^{-2} + |\vec{m}|\big(s^{-1} + \|\vare\|_{L^2_{sol}}\big).
\end{equation}

The differentiation on $s$ of the orthogonality condition in the direction of $Q$ gives
\begin{equation}
\begin{split}
    &0 = \big( \partial_y (\mathcal{L}\vare), Q \big) + \big( (W+F+\vare)^5 - (W+F)^5 - 5Q^4\vare, Q'\big)- \big(\mathcal{E}(W), Q\big)\\
    &+ \big(\vec{m}\cdot\vec{M}\vare, Q\big) + \big(\vec{n}\cdot\vec{M}\vare, Q\big).
\end{split}
\end{equation}

Since $(\Lambda Q, Q) = (Q',Q) = 0$, we get by \eqref{equation_rewriting_of_error} $
\big(\mathcal{E}(W), Q \big) = \big( \mathcal{R}, Q\big)$.
By integration by parts and the orthogonality, we have $
\big(\Lambda \vare,Q  \big) = - \big( \vare, \Lambda Q\big) = 0$.
The kernel of the operator $\mathcal{L}$ yields
$\big( \partial_y(\mathcal{L}\vare) , Q\big) = -\big( \vare, \mathcal{L}Q' \big) = 0$.

Hence, we get
\begin{equation}
    \big( \mathcal{R}, Q\big) = \big( (W+F+\vare)^5 - (W+F)^5 - 5Q^4 \vare, Q' \big) - \Big(\frac{\hatsig_s}{\lambda}-1\Big)\big(\vare, Q'\big) - \frac{\lambda^2}{4}\big( \vare, Q' \big).
\end{equation}
Therefore 
\begin{equation}
    \big|\big(\mathcal{R}, Q \big)\big| \lesssim \|\vare\|_{L^2_{sol}}\big(s^{-1} + |\vec{m}|\big) + \|\vare\|^2_{L^2_{sol}}.
\end{equation}
Combining this estimate of $(\mathcal{R}, Q)$ with \eqref{estimate_on_gs_with_prod_scal_R_Q}, we have
\begin{equation}\label{proof_diff_in_s_of_ortho_first_estim_on_lambda_2_g_s}
    |\lambda^2 \partial_s g| \lesssim \|\vare\|_{L^2_{sol}}\big(s^{-1} + |\vec{m}|\big) + \|\vare\|^2_{L^2_{sol}} +s^{-3} + s^{-1}|\vec{m}| + |b_s|e^{-cs^{\gamma}}.
\end{equation}
The rewriting in \eqref{rewriting_lambda_2_g_s}
and the following estimate, coming from differentiation, 
\begin{equation}
    \Big| \lambda^2 \partial_s\Big(  \frac{4}{\int Q}c_0 \lambda^{-\frac 32}e^{-\frac 12 \hatsig}\Big) \Big| \lesssim s^{-2} +s^{-1}|\vec{m}|,
\end{equation}
it yields for $s_0 \gg 1$
\begin{equation}\label{proof_diff_in_s_of_ortho_first_estim_on_b_s}
    \begin{split}
        |b_s| \lesssim s^{-2} + |\vec{m}| \big( s^{-1} + \|\vare\|_{L^2_{sol}}\big) + \|\vare\|^2_{L^2_{sol}}.
    \end{split}
\end{equation}
Inserting this estimate in \eqref{proof_diff_in_s_of_ortho_first_estim_on_m}, for $s_0\gg1 $ and $\alpha^* \ll 1$ the estimate \eqref{estimate_on_m_mod_estim} holds true.
Finally, inserting this estimate \eqref{estimate_on_m_mod_estim} in \eqref{proof_diff_in_s_of_ortho_first_estim_on_b_s}, we have \eqref{estimate_on_bs_mod_estim}

Combining these estimates on $\vec{m}$ and $b_s$ with \eqref{proof_diff_in_s_of_ortho_first_estim_on_lambda_2_g_s} yields \eqref{estimate_on_gs_mod_estim}.
    
\end{proof}

\subsection{Bootstrap estimates}

Fix two $C^5(\RR)$ functions $\vphi$ and $\psi$ such that
\begin{equation}\label{definition_of_vphi_and_psi}
    \vphi(y)=\begin{cases}
        e^y, \quad y<-1,\\
        1+y, \quad -\frac{1}{2}<y<\frac{1}{2},\\
        y,  \quad y>2,
        \end{cases}
        \qquad  \psi(y) = \begin{cases}
        e^{2y}, \quad y<-1,\\
        1, \quad y>-\frac{1}{2},
    \end{cases}
\end{equation}
and 
\begin{equation}\label{condition_on_the_def_of_vphi_and_psi}
     \vphi'(y)>0, \quad \psi'(y)>0, \quad y \vphi'(y) \leq \vphi(y), \qquad \forall y \in \RR.
\end{equation}

Let $B>100$ to be fixed later. Define on $\RR$
\begin{equation}
    \vphi_B(y):= \vphi\Big(\frac{y}{B} \Big) \qquad \text{and}\qquad \psi_B(y):= \psi\Big(\frac{y}{B} \Big).
\end{equation}

For $s\in [s_0,s^*]$, we introduce a norm with weight
\begin{equation}
    \mathcal{N}_B(s):= \Big(\int(\partial_y \vare)^2(s,y)\psi_B(y) \,dy + \int\vare^2(s,y)\vphi_B(y) \,dy \Big)^{\frac{1}{2}}.
\end{equation}
From the definition \eqref{def_of_L2loc_L2B} of the $L^2_{loc}$-norm, it holds
\begin{equation}\label{relation_L2loc_with_N_B}
    \|\vare\|^2_{L^2_{sol}}+ \|\partial_y \vare\|^2_{L^2_{sol}} + \int \vphi'_B\,\vare^2  \lesssim \mathcal{N}^2_B
    \qquad\text{and}\qquad \|\vare\|^2_{L^2_{sol}} \lesssim B\int \vphi'_B \vare^2 .
\end{equation}

Set $j = \frac 52$ (other values of $j\in (2,3)$ are possible). In addition to \eqref{BS1}, we will work under the following assumptions 
\begin{equation}\tag{BS2}\label{BS2}
        \mathcal{N}^2_B(s) \;\leq\; s^{- j}\,\qquad \text{and} \qquad \|\vare(s)\|_{H^1} \;\leq\; \alpha^*\,\quad\text{for}\quad s \in [s_0,s^*].
\end{equation}

\begin{remark}
    In particular, it allows us to control the full $L^{\infty}$ norm of $\vare$ in terms of $\alpha^*$
    \begin{equation}\label{L_infty_norm_of_vare_control}
        \|\vare(s)\|_{L^{\infty}} \lesssim \|\vare(s)\|_{H^1} \lesssim \alpha^*.
    \end{equation}
\end{remark}


\begin{lemma}[Consequences of bootstrap assumptions]
Assuming $g(s_0) =0$ and \eqref{BS1}, \eqref{BS2}, for $s_0\gg1$, it holds on $[s_0,s^*]$
\begin{equation}\label{conseq_of_BS_on_m_on_bs_and_on_gs}
    |\vec{m}| \lesssim s^{-\frac{5}{4}}, \qquad\qquad  |b_s|\lesssim s^{-2} \qquad,\qquad |\partial_s g| \lesssim s^{-\frac 14},
\end{equation}
\begin{equation}\label{conseq_of_BS_on_g_and_hs}
    \lambda^2 \,| g | \lesssim  s^{-\frac 54} \qquad \text{and}\qquad | \partial_s h | \lesssim s^{-\frac 74},
\end{equation}
\begin{equation}\label{estimate_on_L2_norm_of_grad_vare}
    \|\partial_y \vare(s)\|^2_{L^2} \lesssim \lambda^2 \big|E(U_0)\big| + s^{-\frac 54}.
\end{equation}
\end{lemma}
\begin{proof}
Lemma \ref{lemma_mod_estimates_after_decomp}, \eqref{relation_L2loc_with_N_B}, \eqref{BS1} and \eqref{BS2} yields \eqref{conseq_of_BS_on_m_on_bs_and_on_gs}.

Integrating the estimate on $\partial_s g$ in \eqref{conseq_of_BS_on_m_on_bs_and_on_gs}, we get the first estimate in \eqref{conseq_of_BS_on_g_and_hs}.

By \eqref{estimate_on_hs_with_g} and \eqref{conseq_of_BS_on_m_on_bs_and_on_gs}, we have
\begin{equation}
   |\partial_s h | \lesssim \lambda^{\frac 12}\big( \lambda^2 |g | + |\vec{m}| \big) \lesssim s^{-\frac 74}.
\end{equation}
The estimate \eqref{estimate_on_L2_norm_of_grad_vare_before_BS2} combined with \eqref{BS1},\eqref{BS2} yields \eqref{estimate_on_L2_norm_of_grad_vare}.

\end{proof}

%% file: Energy_estimates.tex
\section{Energy estimates}\label{section_energy_estimates}
We work in the general context of Lemma \ref{lemma_on_decomposition_around_Q} with 
$(\lambda,\sigma, b, \vare)$ satisfying \eqref{BS1}-\eqref{BS2} on some time interval $[s_0,s^*]$ with $s^*\geq s_0$.

We define the mixed energy-virial functional as
    \begin{equation*}
        \mathcal{F} \;=\; \int \Big[ (\partial_y \vare )^2 \psi_B + \vare^2 \varphi_B - \frac{1}{3}\Big( (W+F+\vare)^6-(W+F)^6-6(W+F)^5\vare \Big) \psi_B \Big]\, dy.
    \end{equation*}
(see \cite[Section IV]{1M})

\vspace{0.2cm}
\begin{proposition}(Monotonicity formula on $\mathcal{F}$)\label{Proposition_Monot_formula_on_F}\\ Let $j = \frac 52$. The following estimates hold on $[s_0,s^*]:$
\begin{itemize}
    \item Lyapounov control : There exists $0<\mu<1$ and $C>0$ such that
    \begin{equation}\label{Lyapounov_control_of_F}
    \begin{split}
        s^{-j}\frac{d}{ds}\big[ s^{j}\,\mathcal{F}\big] + \frac{\mu}{8}\int \vphi'_B \big(\vare^2 + (\partial_y \vare)^2\big)  \leq 8 \, \,s^{-1}\int \vphi_B\vare^2  + CB^2\,s^{-4}.
    \end{split}
    \end{equation}
    \item Coercivity of $\mathcal{F} $ :  
    \begin{equation}\label{coercivity_of_mathcal_F}
        \mathcal{F}\;\gtrsim\;\mathcal{N}^2_B.
    \end{equation}
\end{itemize}
\end{proposition}
The proof of the proposition is similar to the proof of \cite[Proposition VI.1]{1M}
\vspace{0.3cm}
\begin{remark}
The positive term $s^{-1}\int\vphi_B\vare^2$ on the right-hand side of \eqref{Lyapounov_control_of_F} prevents us from closing the estimate on $\mathcal{F}$. In \cite{MMPIII}, the authors opt for interpolation using a fixed weighted norm, estimated independently, which implies the constraint $\nu >11/13$ on the blow-up rate. The present work and the paper \cite{1M} introduce the scaling term $\mathcal{J}$, with a flexible weight depending on the perturbing tail $f_0$. This enables a more subtle interpolation, thereby allowing us to reach the full range of blowup rates $\nu \in [1/2,1)$.
\end{remark}

\vspace{0.3cm}
Let $\mu$ be the parameter from Proposition \ref{Proposition_Monot_formula_on_F}, define $\theta$ large enough such that 
\begin{equation}\label{definition_of_1_over_theta}
    0<\frac{1}{\theta} < \frac{\mu}{32\times 8}.
\end{equation}

Define $\Upsilon$ an $C^{1000}(\RR)$ non-decreasing function such that
\begin{equation}\label{definition_of_Upsilon}
    \Upsilon(z) : =\begin{cases}
        0,  \quad \text{for  } z\leq \frac{1}{2\theta},\\
        1,  \quad \text{for  } z \geq \frac{1}{\theta}
    \end{cases} \qquad \text{and} \quad |\Upsilon''|+ |\Upsilon'''| \lesssim \Upsilon^{\frac{110}{111}}, \quad \Upsilon' \geq 0 \quad \text{on}\quad \RR.
\end{equation}
Such a condition on the $\Upsilon$ can be satisfied by a polynomial-type connection, for example.

In order to control the positive term $s^{-1}\int \vphi_B\vare^2$  appearing in the Lyapounov control of $\mathcal{F}$, we introduce another functional, $\mathcal{H}$ with an additional term at a scaling level. We reduce its control to the interval $y \gtrsim s$ thanks to the cut-off function $\Upsilon$.
A similar cut-off could also be introduced in \cite{1M} and \cite{MMPIII}. It would not modify the result, but it would simplify and improve several estimates.

In the case of the present paper, the presence of the cut-off $\Upsilon$ is crucial. The constant $\theta$ is chosen to be large enough to control the left in space part of $s^{-1}\int \vphi_B\vare^2$ by the smoothing term coming from Lyapounov control of $\mathcal{F}$.

Set $\kappa = 4$. (Other values of $\kappa$ are possible). In order to control $s^{-1}\int \vphi_B\vare^2$ on the left in space, we require that 
\begin{equation}\label{condition_on_kappa}
    \kappa > j+1 = \frac 72.
\end{equation}
We take the closest natural choice for $\kappa$. Note that if the value of $\kappa$ is too large, a stronger condition must be imposed on the derivatives of $\Upsilon$ in \eqref{definition_of_Upsilon}.

We consider the functional
\begin{equation*}
        \mathcal{H} \;=\; s^j \, \mathcal{F} + \mathcal{J}
    \end{equation*}
where 
\begin{equation}
    \mathcal{J} := \int \phi_{\lambda,\hatsig}(y)\,\vare^2\,dy \qquad\text{with}\qquad \phi_{\lambda,\hatsig}(y) : = e^{2\kappa \lambda y + \kappa \hatsig}\Upsilon(\lambda y).
\end{equation}

Firstly, we announce control of the derivative of 
the scaling term.
For simplicity, we define the following function
\begin{equation}
    \phi^d_{\lambda,\hatsig} := e^{2\kappa \lambda y + \kappa \hatsig}\Upsilon'(\lambda y).
\end{equation}

\begin{lemma}\label{lemma_control_of_scaling_term}
For all $s\in[s_0,s^*]$ holds
\begin{equation}\label{equation_control_of_scaling_term}
\begin{split}
    \frac{d}{ds}\big[\mathcal{J} \big] + \frac{\kappa \lambda}{2} \int \phi_{\lambda,\hatsig}(y)\, \vare^2 + \frac{\lambda}{2} \int \phi^d_{\lambda,\hatsig}(y)\,\vare^2 \lesssim s^{-102+\kappa}.
\end{split}
\end{equation}
\end{lemma}

The above results enable us to control the evolution of $\mathcal{H}$. 
\begin{proposition}(Monotonicity formula on $\mathcal{H}$)\label{Proposition_Monot_formula_on_H}\\
The following estimate holds on $[s_0,s^*]$:
\begin{equation}\label{Prop_Monotonicity_formula_equation}
    \frac{d}{ds}\big[\mathcal{H}\big] \lesssim s^{j-4}.
\end{equation}
\end{proposition}

\begin{proof}[Proof of Proposition \ref{Proposition_Monot_formula_on_H}]

Assuming Proposition \ref{Proposition_Monot_formula_on_F} and Lemma \ref{lemma_control_of_scaling_term}, we have for some $C>0$
    \begin{equation}\label{proof_monot_H_control_of_H}
        \begin{split}
            \frac{d}{ds}\big[\mathcal{H}\big]+s^j \frac{\mu}{8} \int \vphi'_B \big( \vare^2 + (\partial_y \vare)^2\big) + \frac{\kappa \lambda}{2} \int e^{2\kappa \lambda y + \kappa \hatsig}\Upsilon(\lambda y) \vare^2
            \leq C s^{j-4} + 8 \, \, s^{j-1}\int \vphi_B\,\vare^2.
        \end{split}
    \end{equation}

We decompose
\begin{equation}\label{proof_monot_H_rewriting_of_the_bad_term}
    \begin{split}
        8\, s^{j-1}\int \vphi_B\,\vare^2 = 8\, s^{j-1}\int_{y<(\theta\lambda)^{-1}} \vphi_B\,\vare^2 + 8\, s^{j-1}\int_{y>(\theta\lambda)^{-1}} \vphi_B\,\vare^2.
    \end{split}
\end{equation}

Recall that $B$ was fixed in the proof of Proposition \ref{Proposition_Monot_formula_on_F}. The first term on the right hand side is handled using the following identity. Due to the continuity of $\vphi'$, for $a > a_0$ where 
\begin{equation}
    a_0 \geq \max\Big(2,\frac{2}{\min_{y\in[-1,2]}\vphi'(y)}\Big),
\end{equation} 
it holds $\vphi \leq a \vphi'$ on $(-\infty, a]$. Therefore
\begin{equation}
    \vphi_B \leq \,B\,a\,\vphi'_B \quad \text{on} \quad (-\infty, B\,a].
\end{equation}

We take $a>a_0$ such that $B a = (\theta \lambda)^{-1}$ where $\theta$ is fixed in \eqref{definition_of_1_over_theta}. We observe that for $s_0 \gg 1$, it holds
\begin{equation}\label{proof_monot_H_condition_on_a_for_cut_off}
    a>2 \iff (\theta \lambda)^{-1} > 2B.
\end{equation}
Thus, from the consequence of \eqref{BS1} on $\lambda$ in \eqref{conseq_of_BS1_strict_ineq} and the definition of $\theta$ in \eqref{definition_of_1_over_theta}, we get
\begin{equation}
\begin{split}
    8\, \,s^{j-1}\int_{y<(\theta\lambda)^{-1}}\vphi_B \vare^2 = 8\, \, s^{j-1} \int_{y<Ba} \vphi_B \vare^2
    &\leq 8\, \frac{1}{\theta}\lambda^{-1} \, s^{j-1}\,\int_{y<Ba} \vphi'_B \vare^2\\
    &\leq 2\times8\, \frac{1}{\theta}s^{j}\int \vphi'_B \vare^2 < \frac{\mu}{16}\,s^{j}\int \vphi'_B\vare^2.
\end{split}
\end{equation}
The constant $\frac{\mu}{16}$ allows us to control this term by the one appearing with $\frac{\mu}{8}$ in \eqref{proof_monot_H_control_of_H}.

The second term in \eqref{proof_monot_H_rewriting_of_the_bad_term} is handled as follows. By \eqref{proof_monot_H_condition_on_a_for_cut_off}, we have $\vphi_B(y) = y/B$ on the interval $\{y>(\theta\lambda)^{-1}\}$. Since $\mathbf{1}_{y>(\theta \lambda)^{-1}}\leq \Upsilon(\lambda y)$ on $\RR$ and the definition of $j,k$ yields $j-k+1= -\frac 12<0$. Therefore, by \eqref{condition_on_kappa} we get 
\begin{equation}
\begin{split}
    8\, s^{j-1}\int_{y>(\theta\lambda)^{-1}}\vphi_B \vare^2 &\leq \frac{8}{B} s^{j-1}\int_{y>(\theta\lambda)^{-1}}y \vare^2
    = \frac{4}{B\kappa}\lambda^{-1} \sup_{y>0}\big|2\kappa \lambda y\, e^{-2\kappa \lambda y}\big| \, s^{j-1}\, \int_{y>(\theta\lambda)^{-1}} e^{2\kappa\lambda y} \vare^2 \\
    &\lesssim \lambda^{-2}e^{-\kappa \hatsig}\,s^{j-1}\, \kappa\lambda \int_{y>(\theta\lambda)^{-1}}e^{2\kappa\lambda y + \kappa \hatsig} \vare^2
    \lesssim s^{j-\kappa+1}\,\kappa\lambda\int e^{2\kappa\lambda y + \kappa \hatsig}\,\Upsilon(\lambda y)\, \vare^2\\
    &\leq \frac{1}{2^{10}}\kappa\lambda \int e^{2\kappa\lambda y + \kappa \hatsig}\,\Upsilon(\lambda y)\, \vare^2.
\end{split}
\end{equation}
\end{proof}
Thus, \eqref{proof_monot_H_control_of_H} implies \eqref{Prop_Monotonicity_formula_equation}.

\vspace{0.3cm}
\begin{proof}[Sketch of the proof of the Proposition \ref{Proposition_Monot_formula_on_F}]

\textit{Proof of Lyapounov control on $\mathcal{F}$ \eqref{Lyapounov_control_of_F}.}
The mixed energy-virial functional is identical to the one introduced in the author's previous work (see \cite{1M}), and its control was fully studied in Section IV of that work. In the current paper we use the established control, but our aim is to determine a precise bound on the positive sign term. For this reason, most of the estimates presented are studied in the aforementioned work. Here, we refine the analysis of the ones that affect the scaling term.

\begin{remark}
The following relation will be useful throughout the proof. Let $g_1,g_2 \in H^1$, it holds
\begin{equation}\label{relation_on_partial_eps_and_order_5}
    \partial_y \Bigg( \frac{1}{6} \Big[(g_1+g_2)^6 - g_1^6 - 6 g_1^5 g_2 \Big] \Bigg) = \partial_y g_1 \Big[ (g_1+g_2)^5 - g_1^5 - 5 g_1^4 g_2\Big] + \partial_y g_2 \Big[ (g_1+g_2)^5 - g_1^5 \Big].
\end{equation}
\end{remark}

First, we show that there exists $0<\mu<1$ such that
\begin{equation}\label{mathcalF_estim_on_all_dds}
\begin{split}
    s^{-j}\frac{d}{ds}\big[s^j \mathcal{F} \big]  + 2\int_{y<-\frac B2} \psi'_B\,(\partial^2_y \vare)^2
    +\frac\mu 8 \int \vphi'_B \big( \vare^2 + (\partial_y \vare)^2 \big) + \int_{y<-\frac B2} \psi'_B\,(\partial_y \vare)^2 \leq 8\, s^{-1}\int \vphi_B \,\vare^2 + CB^2 s^{-4}.
\end{split}
\end{equation}

Using the equation \eqref{equation_of_eps_s} of $\vare_s$, we write 
\begin{equation}
    \begin{split}
        &s^{-j}\frac{d}{ds}\big[s^j\, \mathcal{F}\big]  
        = \frac{d}{ds}\big[\mathcal{F}\big]+ \frac js\mathcal{F} = 2\int G_B(\vare)\big( \vare_s - \frac{\lambda_s}{\lambda}\Lambda \vare \big) + 2\frac{\lambda_
        s}{\lambda}\int \Lambda \vare\, G_B(\vare) \\
        & + \frac{j}{s}\mathcal{F} - 2 \int \partial_s(W+F)\, \big[ (W+F+\vare)^5 - (W+F)^5 - 5(W+F)^4\vare \big] \psi_B\\
        & =:\mathcal{F}_{1}+\mathcal{F}_{2}+\mathcal{F}_{3}+\mathcal{F}_{4}+\mathcal{F}_{5},
    \end{split}
\end{equation}
where
\begin{equation}
\begin{split}
    \mathcal{F}_{1} &= 2 \int G_B(\vare) \, \partial_y\Big[ -\partial^2_{y}\vare + \vare - \big( (W+F+\vare)^5 - (W+F)^5 \big)\Big],\\
    \mathcal{F}_{2} &= -2\int G_B(\vare)\mathcal{E}(W),\\
    \mathcal{F}_{3}&= 2 \Big(\frac{\hatsig_s}{\lambda}-1\Big) \int G_B(\vare) \partial_y \vare + \frac{1}{2}\lambda^2 \int G_B(\vare) \partial_y \vare, \\
    \mathcal{F}_{4} &= 2\frac{\lambda_s}{\lambda}\int G_B(\vare)\, \Lambda \vare + \frac{j}{s}\mathcal{F}, \\
    \mathcal{F}_{5} &= -2 \int \psi_B\, \partial_s(W+F) \Big[ (W+F+\vare)^5 - (W+F)^5 - 5(W+F)^4\vare\Big]
\end{split}
\end{equation}
and
\begin{equation}
    G_B(\vare) = -\partial_y(\partial_y \vare \, \psi_B) + \vare \, \vphi_B - \big( (W+F+\vare)^5 - (W+F)^5\big)\psi_B.
\end{equation}

Each $\mathcal{F}_i$ is estimated separately, these estimates are detailed in \cite[Proof of the Proposition IV.1]{1M}. The regime remains the same, but the equation of $\vare$ differs with the appearance of $\hatsig + \frac{\tau(s)}{4}$, which is controlled in the quantity $\mathcal{F}_{3}$.

\textit{Estimate on $\mathcal{F}_{1}$:} 
The term $\mathcal{F}_{1}$ is studied separately on two domains $\{ y > -\frac{B}{2}\}$ and $\{ y < -\frac{B}{2}\}$.

\textit{Estimate of $\mathcal{F}_{1}$ on $(-\frac{B}{2}, +\infty)$:} The following virial estimate holds true on the announced domain. The proof of this result can be found in \cite[Appendix A]{1M} and its variants appear in \cite[Lemma 3.5]{Combet-Martel-17}, \cite[Lemma 3.4]{MMPI}. This proof is based on the coercivity property of a virial quadratic form under suitable repulsivity properties introduced in \cite[Proposition 4]{Martel-Merle-00}.

\begin{lemma}[Localised virial estimate]\label{lemma-localized-virial-estimate}
    There exists $0<\mu<1$ such that 
    \begin{equation}\label{estimate-localized-virial}
        \int_{y>-\frac B2} \big[ 3(\partial_y\vare)^2 + \vare^2 - 5 Q^4 \vare^2 + 20y Q' Q^3 \vare^2\big] \geq \frac{\mu}{2}\int_{y>-\frac B2} \big( \vare^2 + (\partial_y \vare)^2 \big) - \frac{C}{\mu}\frac{1}{B^5}\int \vare^2\,e^{-\frac{|y|}{2}}.
    \end{equation}
\end{lemma}

We denote $\mathcal{F}^{(>)}_1$ the restriction of $\mathcal{F}_{1}$ on $(-\frac{B}{2}, +\infty)$.
From the localised virial estimate, it holds
\begin{equation}
    \mathcal{F}^{(>)}_1 -\frac{\mu}{2}\int_{y>-\frac B2}\vphi'_B\,\big(\vare^2 + (\partial_y \vare)^2 \big) \leq \frac{\mu}{2^{20}}\int \vphi'_B\,\vare^2 + \mathcal{F}^{(>)}_{1,3,3}.
\end{equation}
The term $\mathcal{F}^{(>)}_{1,3,3}$ is given by
\begin{equation}
\begin{split}
    \mathcal{F}^{(>)}_{1,3,3} = 20 \int_{y>-\frac B2} (\vphi_B - \psi_B)\vare^2 \big[ Q'Q^3 - \partial_y(W+F)(W+F)^3 \big].
\end{split}
\end{equation}

We use the following rewriting
\begin{equation}
    \partial_y(W+F)(W+F)^3 - Q'Q^3 = \partial_y(W+F)\big( (W+F)^3 - Q^3\big) + Q^3 \partial_y (W-Q+F).
\end{equation}
Therefore, since $|\vphi_B - \psi_B| \lesssim \vphi_B $, we get
\begin{equation}\label{proof_monot_of_F_rewr_of_F133}
\begin{split}
     \mathcal{F}^{(>)}_{1,3,3} \lesssim \int_{y>-\frac B2} \vphi_B\, \vare^2 \,|\partial_y(W+F)|\, \big|(W+F)^3 - Q^3\big| + \int_{y>-\frac B2} \vphi_B\, \vare^2 Q^3 \big|\partial_y(W-Q+F)\big|.
\end{split}
\end{equation}
To treat the appearing terms, we use the following relation
\begin{equation}\label{proof_monot_H_decr_exp_phi_less_phi_d}
    e^{-\frac{|y|}{2}}\vphi_B \lesssim B \vphi'_B, \qquad \forall y \in \RR.
\end{equation}

For the first term on r.h.s. in \eqref{proof_monot_of_F_rewr_of_F133}, by interpolation, the computation $P'_b = P'\chi_b + b^{\gamma}$, estimates in Lemma \ref{lemma_with_all_estimates_on_r_F_etc} and since $Q',R',P' \in \mathcal{Y}$ with \eqref{proof_monot_H_decr_exp_phi_less_phi_d}, we get
\begin{equation}
    \begin{split}
        &\int_{y>-\frac B2} \vphi_B\, \vare^2 \,|\partial_y(W+F)|\, \big|(W+F)^3 - Q^3\big| \lesssim \int_{y>-\frac B2} \vphi_B\, \vare^2 |\partial_y(W+F)|Q^2 \big|bP_b + rR+F\big|\\
        &+ \int_{y>-\frac B2} \vphi_B\, \vare^2 \big|Q'+rR'+bP'\chi_b \big|\,\big|bP_b + rR +F\big|^3
        +\int_{y>-\frac B2} \vphi_B\,\vare^2 \big|b^{1+\gamma}\chi'(b^{\gamma} y)P + \partial_y F\big|\,\big| bP_b + rR + F \big|^3\\
        &\lesssim  Bs^{-\frac 32}_0 \int_{y>-\frac B2} \vphi'_B\vare^2 + (s^{-1-\gamma} + s^{-\frac 32})(s^{-1} + s^{-\frac 12})^3 \int_{y>-\frac B2}\vphi_B\vare^2\\
        &\lesssim Bs^{-\frac 32}_0 \int_{y>-\frac B2} \vphi'_B\vare^2 + s^{-1}_0\,s^{-1} \int_{y>-\frac B2}\vphi_B\vare^2.
    \end{split}
\end{equation}

The second term on the r.h.s. in \eqref{proof_monot_of_F_rewr_of_F133} is handled with \eqref{proof_monot_H_decr_exp_phi_less_phi_d}.
Therefore, using the exponential decay of $Q$, we get
\begin{equation}
    \begin{split}
        \int_{y>-\frac B2} \vphi_B\, \vare^2 Q^3 \big|\partial_y(W-Q+F)\big| \lesssim \int_{y>-\frac B2} \vphi'_B\,\vare^2\big( |b\,P'_b| + |r \partial_y R| + e^{-2|y|}|\partial_y F|\big)
        \lesssim Bs^{-1}_0\,\int_{y>-\frac B2}\vphi'_B\,\vare^2.
    \end{split}
\end{equation}

Choosing $s_0(B, \mu)\gg1$, we have
\begin{equation}\label{mathcalF_estim_onF1_on_right}
    \mathcal{F}_{1}^{(>)} + \frac\mu 4 \int_{y>-\frac B2} \vphi'_B\, \big( \vare^2 + (\partial_y \vare)^2 \big) \leq \frac{1}{2^{10}}s^{-1}\int \vphi_B\, \vare^2.
\end{equation}

\textit{Estimate on $\mathcal{F}_{1}^{(<)}$, domain $\{ y < -\frac{B}{2}\}$:}
\begin{equation}\label{mathcalF_estim_onF1_on_left}
    \mathcal{F}_{1}^{(<)}  \leq - 2 \int_{y<-\frac B2} \psi'_B \,(\partial^2_y \vare)^2 - \frac{\mu}{4}\int_{y<-\frac B2} \vphi'_B\, \big( \vare^2+(\partial_y \vare)^2 \big) -
    \int_{y<-\frac B2} \psi'_B \, (\partial_y \vare)^2.
\end{equation}

\textit{Estimate on $\mathcal{F}_{2}$:} 

\begin{equation}\label{mathcalF_estim_onF2}
     \mathcal{F}_{2}  \leq C B^2 \, s^{-4} + \frac{\mu}{2^{10}}\int \vphi'_B \big( \vare^2 + (\partial_y \vare)^2 \big).
\end{equation}

\textit{Estimate on $\mathcal{F}_{3}$:} As in \cite{1M}, we get
\begin{equation}
    \Big|\int G_B(\vare)\partial_y \vare \Big| \lesssim B \int \vphi'_B \big(\vare^2 + (\partial_y \vare)^2 \big).
\end{equation}
Therefore
\begin{equation}\label{mathcalF_estim_onF3}
    \begin{split}
     \mathcal{F}_{3} \lesssim (|\vec{m}| + \lambda^2)\, B\, \int \vphi'_B \big( \vare^2 + (\partial_y\vare)^2 \big)  \lesssim s^{-\frac{5}{4}} \, B\, \int \vphi'_B \big( \vare^2 + (\partial_y\vare)^2 \big)
     \leq \frac{\mu}{2^{10}}\int \vphi'_B \big( \vare^2 + (\partial_y\vare)^2 \big).
     \end{split}
\end{equation}

\textit{Estimate on $\mathcal{F}_{4}$:}
It holds
\begin{equation}
    \mathcal{F}_4 + \frac{\lambda_s}{\lambda}\int y\vphi'_B\, \vare^2 - \frac{j}{s}\int \vphi_B\,\vare^2 \leq Cs^{-1}_0 B\int \vphi'_B\,\big(\vare^2+ (\partial_y \vare)^2\big) + \mathcal{F}_{4,3},
\end{equation}
where
\begin{equation}
    \mathcal{F}_{4,3} = 2\frac{\lambda_s}{\lambda}\int \psi_B\,\Lambda(W+F)\,\big( (W+F+\vare)^5 - (W+F)^5 - 5(W+F)^4 \vare \big).
\end{equation}
By interpolation, using the definition of $\Lambda$ and since $\psi_B \lesssim B \vphi'_B$ on $\RR$, we get the following estimate
\begin{equation}
    \begin{split}
        &\mathcal{F}_{4,3} \lesssim B \Big|\frac{\lambda_s}{\lambda}\Big| \int \vphi'_B \,\big( |W+F| + y|\partial_y(W+F)|\big)\big( |W+F|^3 \vare^2 + |\vare|^5\big)\\
        &\lesssim s^{-1}_0B\int \vphi'_B\,\vare^2 + B\Big|\frac{\lambda_s}{\lambda}\Big|\int y \vphi'_B \big( \big|Q'+bP'\chi_b + rR'\big| + \big| b^{1+\gamma}P\chi'(b^{\gamma}y) + \partial_y F \big|\big)\big( |W+F|^3 \vare^2 + |\vare|^5\big)\\
        &\lesssim s^{-1}_0B\int \vphi'_B\,\vare^2 + s^{-1}_0\,s^{-1}\int y\vphi'_B\,\vare^2.
    \end{split}
\end{equation}
We have also used Lemma \ref{lemma_with_all_estimates_on_r_F_etc}, the decay $|y|(|Q'|+|P'|+|R'|)\vphi'_B \lesssim \vphi'_B $ on $\RR$. 

By the definition of $\vphi$ in \eqref{condition_on_the_def_of_vphi_and_psi}, we have $y\vphi'_B \leq \vphi_B$. Thus, since by \eqref{conseq_of_BS1_strict_ineq}, it holds $\big|\frac{\lambda_s}{\lambda}\big| \leq 2s^{-1}$, we have
\begin{equation}\label{mathcalF_estim_onF4}
\begin{split}
    &\mathcal{F}_{4} \leq CBs^{-1}_0\,\int \vphi'_B\,\big(\vare^2+ (\partial_y \vare)^2\big) + 2\Big|\frac{\lambda_s}{\lambda}\Big| \int y\vphi'_B \,\vare^2 + \frac{j}{s}\int \vphi_B\,\vare^2\\
    &\leq (4 + j) s^{-1}\, \int \vphi_B\, \vare^2 + \frac{\mu}{2^{10}}\,\int \vphi'_B\,\big(\vare^2+ (\partial_y \vare)^2\big).
\end{split}
\end{equation}

\textit{Estimate on $\mathcal{F}_{5}$:}
\begin{equation}\label{mathcalF_estim_onF5}
    \begin{split}
        \mathcal{F}_{5} \lesssim Bs^{-1}_0\,\int \vphi'_B\,\vare^2 + Bs^{-\frac32}_0 \,\frac{\lambda_s}{\lambda}\int \vphi_B\,\vare^2
        \leq \frac{1}{2^{10}}s^{-1}\int \vphi_B\,\vare^2 + \frac{\mu}{2^{10}}\int \vphi'_B\, \vare^2. 
    \end{split}
\end{equation}

Finally, combining the estimates \eqref{mathcalF_estim_onF1_on_right}, \eqref{mathcalF_estim_onF1_on_left}, \eqref{mathcalF_estim_onF2}, \eqref{mathcalF_estim_onF3}, \eqref{mathcalF_estim_onF4}, \eqref{mathcalF_estim_onF5}, yields the estimate announced in \eqref{mathcalF_estim_on_all_dds}.

\vspace{0.4cm}

\textit{Proof of the coercivity of $\mathcal{F}$ \eqref{coercivity_of_mathcal_F}:}\\
The functional $\mathcal{F}$ is rewritten as $\mathcal{F} = \mathcal{F}_{I} + \mathcal{F}_{II}$ with 
\begin{equation}
    \begin{split}
        &\mathcal{F}_{I} =  \int \Big[ (\partial_y \vare)^2\psi_B +\vare^2 \vphi_B - 5Q^4 \vare^2\psi_B \Big],\\
        &\mathcal{F}_{II} = -\frac{1}{3}\int \psi_B \Big[\big( W+F+\vare \big)^6 - (W+F)^6 - 6(W+F)^5\vare - 15Q^4\vare^2 \Big].
    \end{split}
\end{equation}
The term $\mathcal{F}_{I}$ is treated as in the proof of the localised virial estimate \eqref{estimate-localized-virial} which can be found in \cite[Appendix A]{1M}. Other possible references are \cite[Appendix A]{Martel-Merle-02} and \cite[Lemma 3.5 ]{Combet-Martel-17}. Hence, there exists $\eta>0$ such that for $B$ large enough, it holds
\begin{equation}
     \mathcal{F}_1 \geq \eta\, \mathcal{N}^2_B.
\end{equation}

The term $\mathcal{F}_{II}$ can be rewritten as follows
\begin{equation}
    \begin{split}
        &\mathcal{F}_2 = 
         -\frac{1}{3}\int \psi_B \Big[ \big( W+F+\vare\big)^6 - \big( W+F\big)^6 - 6 \big( W+F\big)^5\vare - 15\big( W+F\big)^4\vare^2 \Big]\\
        & +5\int \psi_B\Big[\big( W+F \big)^4 - Q^4 \Big]\vare^2.
     \end{split}
\end{equation}
The estimates in Lemma \ref{lemma_with_all_estimates_on_r_F_etc} and \eqref{BS1} yields
\begin{equation}
    \big| \mathcal{F}_2 \big| \leq \frac{\eta}{2^{10}}\,\mathcal{N}_B^2.
\end{equation}
This completes the proof of the coercivity of $\mathcal{F}$ and proof of Proposition \ref{Proposition_Monot_formula_on_H}.
\end{proof}

\vspace{0.3cm}
\begin{proof}[Proof of Lemma \ref{lemma_control_of_scaling_term}]
We write
\begin{equation}\label{proof_scal_term_first_eq_derivation}
    \frac{d}{ds}\big[ \mathcal{J}\big] = \int \partial_s \Big( e^{2\kappa \lambda y + \kappa \hatsig}\Upsilon(\lambda y) \Big) \vare^2 + 2 \int \vare_s \, \vare \, e^{2\kappa \lambda y + \kappa \hatsig}\Upsilon(\lambda y)=: Z_1 + Z_2.
\end{equation}

The first term is rewritten as follows
\begin{equation}
    Z_1 = 2\kappa \lambda_s \int y\, \phi_{\lambda,\hatsig}\, \vare^2 + \lambda_s \int y\, \phi^d_{\lambda,\hatsig}\, \vare^2 + \kappa \lambda \Big( \frac{\hatsig_s}{\lambda}-1\Big)\int \phi_{\lambda,\hatsig}\, \vare^2 + \kappa \lambda \int \phi_{\lambda,\hatsig}\,\vare^2.
\end{equation}
The equation of $\vare$ in \eqref{equation_of_eps_s} gives
\begin{equation}\label{proof_scal_term_Z2}
\begin{split}
    Z_2 = &2\int \phi_{\lambda,\hatsig}\, \vare \, \Big( \partial_y[-\partial^2_y \vare + \vare] + \frac{\lambda_s}{\lambda}\Lambda \vare\Big)
    + 2 \Big[\Big( \frac{\hatsig_s}{\lambda}-1\Big)+ \frac{\lambda^2}{2}\Big] \int \phi_{\lambda,\hatsig}\, \partial_y \vare\, \vare\\
    &- 2 \int \phi_{\lambda,\hatsig}\, \vare\, \partial_y\big[ (W+F+\vare)^5 - (W+F)^5 \big]
    -2 \int \phi_{\lambda,\hatsig}\, \vare\, \mathcal{E}(W)=: Z_{2,1} + Z_{2,2}-Z_{2,3}-Z_{2,4}.
\end{split}
\end{equation}
The terms $Z_{2,1}, Z_{2,2}$ will compensate the terms in $Z_1$ and the terms $Z_{2,3}, Z_{2,4}$ are estimated thanks to the cut-off $\Upsilon$ and the decay of $Q$.
We have
\begin{equation}
    Z_{2,1} = -\int \phi_{\lambda,\hatsig}' \, \vare^2 - \frac{\lambda_s}{\lambda}\int y \phi_{\lambda,\hatsig}'\, \vare^2 -3\int \phi_{\lambda,\hatsig}'\, (\partial_y \vare)^2 + \int \phi_{\lambda,\hatsig}'''\, \vare^2.
\end{equation}

Calculating $\phi_{\lambda,\hatsig}'$, we get the following equality
\begin{equation}
\begin{split}
    Z_1 -\int \phi_{\lambda,\hatsig}' \, \vare^2 - \frac{\lambda_s}{\lambda}\int y \phi_{\lambda,\hatsig}'\, \vare^2 - 
    = -\kappa \lambda \int \phi_{\lambda,\hatsig}\, \vare^2 - \lambda \int \phi^d_{\lambda,\hatsig}\, \vare^2 + \kappa \lambda\Big(\frac{\hatsig_s}{\lambda}-1\Big) \int \phi_{\lambda,\hatsig}\, \vare^2.
\end{split}
\end{equation}
The second term in \eqref{proof_scal_term_Z2} can be rewritten by integration by parts as
\begin{equation}
    Z_{2,2} = - \Big( \frac{\hatsig_s}{\lambda} -1 \Big)\int \phi_{\lambda,\hatsig}'\, \vare^2 - \frac{\lambda^2}{2}\int \phi_{\lambda,\hatsig}'\,\vare^2.
\end{equation}
Therefore
\begin{equation}
\begin{split}
    &Z_1 + Z_{2,1} + Z_{2,2} = -6 \kappa \lambda \int \phi_{\lambda,\hatsig} (\partial_y \vare)^2 - (\kappa \lambda + \kappa \lambda^3)\int \phi_{\lambda,\hatsig}\, \vare^2\\
    &-3\lambda \int \phi^d_{\lambda,\hatsig}\, (\partial_y \vare)^2 - (\lambda +\frac12 \lambda^3)\int \phi^d_{\lambda,\hatsig}\, \vare^2 +\int \phi'''_{\lambda,\hatsig}\, \vare^2\\
    &- \kappa\lambda\Big( \frac{\hatsig_s}{\lambda}-1 \Big)\int \phi_{\lambda,\hatsig}\, \vare^2 - \lambda \Big(\frac{\hatsig_s}{\lambda}-1 \Big)\int \phi^d_{\lambda,\hatsig}\, \vare^2.
\end{split}
\end{equation}
The terms with $\lambda \Big( \frac{\hatsig_s}{\lambda} -1\Big)$ can be hidden in the terms with $\lambda$ and $\phi_{\lambda,\hatsig}$ or $\phi^d_{\lambda,\hatsig}$, using $|\vec{m}|$ and the presence of $\lambda$.

The term with $\phi'''_{\lambda,\hatsig}$ presents $\lambda^3\ll \lambda$, thanks to $\lambda^2 \lesssim s^{-2}$ some terms can be directly hidden as previously. We have
\begin{equation}
    \Big| \int \phi'''_{\lambda,\hatsig}\, \vare^2 \Big| \lesssim \lambda^3 \Big( \int \phi_{\lambda,\hatsig}\,\vare^2 + \int \phi^d_{\lambda,\hatsig}\,\vare^2 \Big) +  \int e^{2\kappa \lambda y +\kappa \sigma}\, \big(\Upsilon'' + \Upsilon'''\big)(\lambda y)\,\vare^2.
\end{equation}
The term presenting $\Upsilon''(\lambda y)$ is handled the following way, using the condition \eqref{definition_of_Upsilon} on the flatness of the $\Upsilon$. We use as well that $\supp((\Upsilon)^{\frac{1}{10}}(\lambda y)) \subset \supp(\Upsilon'(\lambda y )) \subset [(2\theta\lambda)^{-1},  (\theta\lambda)^{-1}]$. 
\begin{equation}
    \begin{split}
        &\lambda^3 \int e^{2\kappa \lambda y + \kappa \hatsig}| \Upsilon''|(\lambda y)\, \vare^2 \lesssim \lambda^3 \Big( \int e^{2\kappa\lambda y + \kappa \hatsig}|\Upsilon''|^{\frac{1}{10}}(\lambda y)\,\vare^2 \Big)^{\frac{1}{100}}\Big(\int e^{2\kappa\lambda y + \kappa \hatsig}\, |\Upsilon''|^{\frac{999}{990}}(\lambda y)\,\vare^2 \Big)^{\frac{99}{100}}\\
        &\lesssim \lambda^2 \lambda^{100}\int e^{2\kappa \lambda y +\kappa \hatsig} |\Upsilon''|^{\frac{1}{10}}(\lambda y)\,\vare^2 + \lambda^2 \int e^{2\kappa \lambda y +\kappa \hatsig}\, |\Upsilon''|^{\frac{999}{990}}(\lambda y)\, \vare^2\\
        &\lesssim s^{-102+\kappa}\int^{(2\theta\lambda)^{-1}}_{(\theta \lambda)^{-1}} e^{2\kappa \lambda y}\vare^2 + \lambda^2 \int \phi_{\lambda,\hatsig}\,\vare^2 \lesssim \delta(\alpha^*)s^{-102 + \kappa} + \lambda^2 \int \phi_{\lambda,\hatsig}\,\vare^2.
    \end{split}
\end{equation}
The term with $\Upsilon'''(\lambda y)$ is handled similarly.

Therefore, for some $C>0$
\begin{equation}
\begin{split}
     &Z_1 + Z_{2,1} + Z_{2,2} \leq -6 \kappa \lambda \int \phi_{\lambda,\hatsig} (\partial_y \vare)^2 - \frac{3\kappa \lambda}{4} \int \phi_{\lambda,\hatsig}\, \vare^2\\
    &-3\lambda \int \phi^d_{\lambda,\hatsig}\, (\partial_y \vare)^2 - \frac{3\lambda}{4}\int \phi^d_{\lambda,\hatsig}\, \vare^2 + C\,s^{-102+\kappa}.
\end{split}
\end{equation}

By integration by parts the term $Z_{2,3} $ is rewritten as
\begin{equation}
    -Z_{2,3} =  2 \int \phi'_{\lambda,\hatsig}\, \vare\, \big[ (W+F+\vare)^5 - (W+F)^5 \big] + 2\int \phi_{\lambda,\hatsig}\, (\partial_y \vare) \, \big[ (W+F+\vare)^5 - (W+F)^5 \big].
\end{equation}

The first term on the right hand side is handled as follows
\begin{equation}\label{proof_scaling_term_ref_term_s100}
\begin{split}
    &\Big| \int \phi'_{\lambda,\hatsig}\, \vare\, \big[ (W+F+\vare)^5 - (W+F)^5 \big] \Big| \lesssim
     \lambda \int \big(\phi_{\lambda,\hatsig}+\phi^d_{\lambda,\hatsig} \big) |W+F|^4\, \vare^2 + \lambda \int \big(\phi_{\lambda,\hatsig}+\phi^d_{\lambda,\hatsig} \big)\, \vare^6\\
    &\lesssim  \lambda \big(r^4+ \|F\|^4_{L^{\infty}} + \|\vare\|^4_{L^{\infty}} \big)\int \big(\phi_{\lambda,\hatsig}+\phi^d_{\lambda,\hatsig} \big)\, \vare^2 + \lambda\int_{y>(2\theta\lambda)^{-1}} e^{2\kappa \lambda y + \kappa \hatsig} Q^4_b \,\vare^2\\
    &\lesssim \lambda\, \delta(\alpha^*,s^{-1}_0) \int \big(\phi_{\lambda,\hatsig}+\phi^d_{\lambda,\hatsig} \big)\, \vare^2 + \lambda\,\int_{y>(2\theta\lambda)^{-1}} e^{2\kappa \lambda y + \kappa \hatsig} Q^4_b \,\vare^2.
\end{split}
\end{equation}
The term presenting the $Q_b$ is handled using the decay for $y> 0 , |Q_b(y)| \lesssim e^{-\frac{y}{2}}$, for $s_0 \gg1$ it holds 
\begin{equation*}
    \lambda\int_{y>(2\theta\lambda)^{-1}} e^{2\kappa \lambda y + \kappa \hatsig} Q^4_b \,\vare^2 \lesssim \lambda\, s^{\kappa}\int_{y>(2\theta\lambda)^{-1}} e^{2\kappa \lambda y - 2 y}\, \vare^2 \lesssim \lambda\, s^{\kappa}\, e^{-c \lambda^{-1}}\|\vare\|^2_{L^2_{sol}} \lesssim s^{-100}.
\end{equation*}

The second term on the r.h.s. of $Z_{2,3}$ is rewritten with \eqref{relation_on_partial_eps_and_order_5}, we get
\begin{equation}
\begin{split}
    &\int \phi_{\lambda,\hatsig}\, (\partial_y \vare) \, \big[ (W+F+\vare)^5 - (W+F)^5 \big] = \frac{1}{6} \int \phi_{\lambda,\hatsig} \partial_y \big[ (W+F+\vare)^6 - (W+F)^6 - 6(W+F)^5 \vare \big]\\
    &- \int \phi_{\lambda,\hatsig} \partial_y(W+F)\,\big[(W+F+\vare)^5 -(W+F)^5 - 5(W+F)^4\vare \big]\\
    &\lesssim \int \phi'_{\lambda,\hatsig} \big[|W+F|^4\vare^2 + \vare^6\big] + \int \phi_{\lambda,\hatsig} \big( |\partial_y W| + |\partial_y F|\big)\,\big[|W+F|^3\vare^2 + |\vare|^5 \big].
\end{split}
\end{equation}
The first term on the r.h.s. of the inequality above is handled as in \eqref{proof_scaling_term_ref_term_s100}, is therefore bounded by $s^{-100}$.

The second term is managed with the decay $|\partial_y W| \lesssim e^{-\frac{y}{2}}$ for $y> 0$, the bounds $\|W+F\|_{L^{\infty}_y} \lesssim 1$ and the bound of $|\partial_y F|$ in \eqref{estimate_on_L_infty_of_dj_F}. For $s_0 \gg 1$ it holds 
\begin{equation}
\begin{split}
    &\int \phi_{\lambda,\hatsig} \big( |\partial_y W| + |\partial_y F|\big)\,\big[|W+F|^3\vare^2 + |\vare|^5 \big]\\
    &\lesssim \int \phi_{\lambda,\hatsig}\,|\partial_y W| \vare^2 + \|\vare\|^{3}_{L^{\infty}}\int \phi_{\lambda,\hatsig}\,|\partial_y W|\vare^2 + \int \phi_{\lambda,\hatsig}\, |\partial_y F|\, \vare^2 + \|\vare\|^3_{L^{\infty}}\int \phi_{\lambda,\hatsig} |\partial_y F| \vare^2\\
    &\lesssim \int_{y>\lambda^{-1}} e^{2\kappa \lambda +\kappa \hatsig} e^{-\frac{y}{2}}\,\vare^2 + \lambda^{\frac12}\lambda \int \phi_{\lambda,\hatsig} \,\vare^2\\
    &\lesssim s^{\kappa} e^{-c\lambda^{-1}}\|\vare\|^2_{L^2_{sol}} + s^{-\frac 12} \lambda  \int \phi_{\lambda,\hatsig}\, \vare^2
    \lesssim s^{-100} +s^{-\frac 12}_0 \lambda  \int \phi_{\lambda,\hatsig}\, \vare^2.
    \end{split}
\end{equation}

The term $Z_{2,4}$ is rewritten as
\begin{equation}
    Z_{2,4} =- 2 \int \vec{m} \cdot \vec{M}\, Q\, \phi_{\lambda,\hatsig}\, \vare + 2 \int \mathcal{R}\, \phi_{\lambda,\hatsig}\, \vare.
\end{equation}
Since $\Upsilon(\lambda y) \leq \mathbf{1}_{(y>(2\theta\lambda)^{-1})}$, by ${\lambda^{-1}} \gtrsim s$ and the exponential decay, for ${s_0}\gg1$ it holds
\begin{equation}
    \begin{split}
     \Big| \int \vec{m} \cdot \vec{M}\, Q\, \phi_{\lambda,\hatsig}\, \vare \Big| \lesssim |\vec{m}|\,s^{\kappa}\, \int_{y>(2\theta\lambda)^{-1}}e^{2\kappa \lambda y} e^{-\frac 12 y} |\vare| \lesssim |\vec{m}|\, s^\kappa\, \Big( \int_{y>0} e^{-\frac{1}{10}y} \vare^2 \Big)^{\frac 12}\Big(\int_{y > (2\theta \lambda)^{-1}} e^{2\kappa \lambda y -\frac{1}{10}y} \Big)^{\frac 12}\\
     \lesssim s^\kappa\,|\vec{m}|\, \|\vare\|_{L^2_{sol}} e^{-cs} \lesssim s^{-100}.
    \end{split}
\end{equation}

For $s_0(B)\gg 1$, the following holds true
\begin{equation}
    \begin{split}
        \Big| \int \mathcal{R}\,\phi_{\lambda,\hatsig}\, \vare  \Big| \lesssim s^\kappa\, \Big( \int \mathcal{R}^2 \, e^{\frac{y}{B}}\Big)^{\frac 12}\Big( \int_{y> (2\theta\lambda)^{-1}}e^{4\kappa \lambda y - \frac{y}{B}} \, \vare^2 \Big)^{\frac 12} \lesssim s^\kappa \, \|\mathcal{R}\|_{L^2_B} e^{-cs}\|\vare\|_{L^2}\lesssim s^{-100}.
    \end{split}
\end{equation}

Combining all the estimates, \eqref{proof_scal_term_first_eq_derivation} yields a precise control  \begin{equation}\label{precise_control_of_scaling_term}
\begin{split}
    &\frac{d}{ds}\Big[ \int \phi_{\lambda,\hatsig}\,\vare^2\,dy \Big] + 6\kappa \lambda\int \phi_{\lambda,\hatsig}\, (\partial_y \vare)^2 + 3\lambda \int \phi^d_{\lambda,\hatsig} (\partial_y \vare)^2
    +\frac{\kappa \lambda}{2} \int \phi_{\lambda,\hatsig}\, \vare^2\\
    &+ \frac{\lambda}{2} \int \phi^d_{\lambda,\hatsig}\,\vare^2 \lesssim s^{-102+\kappa}.
\end{split}
\end{equation}
\end{proof}


%% file: Construction_of_exploding_solution.tex
\section{Construction of bubbling solutions}\label{section_construction}
 
Fix the final bootstrap assumption as
\begin{equation}\tag{BS3}\label{BS3}
    \lvert h(s) \rvert \leq s^{-\frac12 - 4\rho}.
\end{equation}
We adopt the theory developed in Sections \ref{section_decomposition} and \ref{section_energy_estimates}. We state the general technical result in the rescaled setting, which yields the general result stated in Theorem \eqref{Theorem_principal_result}. 

\begin{proposition}\label{Prop_on_blow_up}
Let $x_0>0$ and $s_0>0$ be large enough and fix $\hatsig_0$ such that
\begin{equation}\label{init_cond_hatsig_b}
    \big| e^{\hatsig_0} - s_0 \big| \leq s_0^{1-4\rho}
\end{equation}
Let $\vare_0 \in H^1(\RR)$ be such that 
\begin{equation}\label{conditons_ortho_and_decroissance_on_vare_0}
s_0^{j} \, \|\varepsilon_0\|^2_{H^1} +s^{j}_0 \int_{y>0} y \vare^2 + s^{\kappa}_0 \int_{y > (3\theta)^{-1}s_0} e^{y}\,\vare^2_0 <s^{-1}_0\, , \quad (\varepsilon_0,y\Lambda Q) = (\varepsilon_0, \Lambda Q) = (\varepsilon_0, Q) = 0\, .
\end{equation}
Then, there exists 
    \begin{equation}\label{init_cond_in_BS}
        \lambda_0 \in [\lambda^-_0,\lambda^+_0], \quad \text{ where } \quad \lambda^{\pm}_0 := \Big(-\frac{4}{\int Q}c_0\, e^{-\frac 12 \hatsig_0} \pm s^{-\frac 12 -4\rho}_0\Big)^2,
    \end{equation}
such that the solution $U$ of \eqref{gKdV_principal_eq} evolving from the initial data
\begin{equation}
    U_0(x):= \lambda_0^{-\frac{1}{2}}\Big( Q_{b_0}+\lambda_0^{-\frac{1}{2}}f(\tau(s_0),\sigma_0)R + \vare_0 \Big)\Big( \frac{x-\sigma_0}{\lambda_0}\Big) + f(\tau(s_0),x)
\end{equation}
with 
\begin{equation}
    b_0 = -\frac{4c_0}{\int Q} \lambda^{\frac 12}_0 e^{-\frac{\hatsig_0}{2}},
\end{equation}
has a decomposition as in Lemma \ref{lemma_on_decomposition_around_Q} and satisfies \eqref{BS1}, \eqref{BS2} and \eqref{BS3} on the interval $[s_0,+\infty)$.

\end{proposition}

\begin{proof}[Proof of the Proposition \ref{Prop_on_blow_up}] 

For $\lambda_0 \in [\lambda_0^-,\lambda_0^+]$, we define
\begin{equation}
    s^*(\lambda_0) :=\sup\{s\geq s_0, \text{ Lemma \ref{lemma_on_decomposition_around_Q}  applies and \eqref{BS1}-\eqref{BS2}-\eqref{BS3} hold on $[s_0,s]$} \}.
\end{equation}

Note that the supremum $s^*$ is determined solely by the saturation of the bootstrap assumptions. The study conducted previously shows that all the estimates in the preceding sections hold on the time interval $[s_0,s^*]$.

For all $\lambda_0 \in [\lambda_0^-,\lambda_0^+]$, we suppose that $s^*(\lambda_0)$ is finite. First, we strengthen the bootstrap assumption on the blow up parameters. The energy estimates conducted in Section \ref{section_energy_estimates} close the weighted norm condition. The proof is finished with a topological argument on the behaviour of the function $h$ for a finite $s^*$, which yields a contradiction.

\vspace{0.2cm}
\textit{Closing \eqref{BS1}.} 
\begin{itemize}
\item $\eqref{BS1}_{\hatsig}:$ Using the condition on $c_0$, we get
\begin{equation}
    \begin{split}
        \big|  \partial_s \big( e^{\hatsig} \big) - 1  \big| = e^{\hatsig} \Big| \hatsig_s - \Big( \frac{4}{\int Q} c_0 \Big)^2 e^{-\hatsig} \Big| \leq e^{\hatsig} \big| \hatsig^{\frac 12}_s  + e^{-\frac 12 \hatsig}\big| \Big( |h(s)| + \big|\hatsig^{\frac 12}_s - \lambda^{\frac 12}\big| \Big).
    \end{split}
\end{equation}
By Lipschitz character of $a \in [c,+\infty)\mapsto a^{\frac 12}$ for any $c>0$, we have
\begin{equation}
    \big| \hatsig^{\frac 12}_s - \lambda^{\frac 12} \big|  = \lambda^{\frac{1}{2}}\Big| \Big(\frac{\hatsig_s}{\lambda}\Big)^{\frac 12} - 1 \Big| \leq \lambda^{\frac 12} |\vec{m}|.
\end{equation}
By \eqref{BS1}, \eqref{BS3}, \eqref{conseq_of_BS_on_m_on_bs_and_on_gs} and the initial condition \eqref{init_cond_hatsig_b}, it follows
\begin{equation}
    \big|  e^{\hatsig}  - s \big| \lesssim s^{1-4\rho}.
\end{equation}

\item $\eqref{BS1}_{\lambda}:$ Since $(4\,c_0/\int Q)^2 = 1$ and \eqref{BS1}, \eqref{BS3}, it holds
\begin{equation}
    \big| \lambda - e^{-\hatsig} \big| = |h(s)|\, \Big| \lambda^{\frac 12}- \Big(\frac{4\, c_0}{\int Q} \Big) e^{-\frac 12 \hatsig} \Big| \lesssim s^{-1-4\rho}.
\end{equation}
By the improved $(BS1_{\hatsig})$, we get
\begin{equation}
    \big| \lambda - s^{-1} \big| \lesssim \big| e^{-\hatsig} - s^{-1} \big| + s^{-1-4\rho} \lesssim s^{-1}\, e^{-\hatsig}\, \big| e^{\hatsig} - s\big| + s^{-1-4\rho} \lesssim s^{-1-4\rho}.
\end{equation}

\item $\eqref{BS1}_b:$ We write (using $c_0 = - \int Q /4$)
\begin{equation}
    \big| s\, b - 1 \big| \leq s\, \lambda^{2}\,\Big( |g(s)| + \big| \lambda^{-\frac 32}e^{-\frac 12 \hatsig}  - s^{-1}\, \lambda^{-2}\big| \Big) \leq s \, \lambda^2 \, \Big( |g(s)| + \lambda^{-\frac 32} \big| e^{-\frac 12 \hatsig} - s^{-1}\, \lambda^{-\frac 12} \big|\Big).
\end{equation}
The $\frac12 -$Holder and Lipshitz far from zero properties of $a:\mathbb{R}^{+} \mapsto a^{\frac12}$ yields
\begin{equation}
\begin{split}
    &\big| e^{-\frac 12 \hatsig}  - s^{-1}\,\lambda^{-\frac 12}\big| \leq \big|e^{-\frac 12 \hatsig} - s^{-\frac 12} \big| + s^{-1}\,\big|\lambda^{-\frac 12}- s^{\frac 12}  \big|\\
    &\leq s^{-\frac 12}\,e^{-\frac 12 \hatsig}\, \big| e^{\frac 12 \hatsig} - s^{\frac 12}\big| + s^{-1}\,\lambda^{-\frac 12} \,\big| \big(s\,\lambda \big)^{\frac 12} - 1 \big|\\
    &\lesssim s^{-1}\,\big| e^{\hatsig} - s \big|^{\frac 12} + s^{\frac 12}\, \big|  \lambda - s^{-1}\big|.
\end{split}
\end{equation}
Thus, the improved conditions $(BS1_\hatsig)$, $(BS1_\lambda)$ and the estimate on $g$ in \eqref{conseq_of_BS_on_g_and_hs} yields
\begin{equation}
    \big| b - s^{-1}\big| \lesssim s^{-2}|g(s)| + s^{-1-2\rho} \lesssim s^{-\frac{11}{4}} + s^{-1 - 2\rho} \lesssim s^{-1-2\rho}.
\end{equation}

\end{itemize}

\vspace{0.2cm}
\textit{Closing \eqref{BS2}.}\\
Proposition \ref{Proposition_Monot_formula_on_H} yields
\begin{equation}
    \frac{d}{ds}\big[ \mathcal{H} \big] \;\lesssim\;   s^{-\frac{3}{2}}.
\end{equation}

Integration on $[s_0,s]$ for some $s\leq s^*$ and positivity of the scaling term, yields
\begin{equation}
    s^{j}\mathcal{F}(s)\leq \mathcal{H}(s) \lesssim s^{-\frac 12}_0 + s^{j}_0 \mathcal{F}(s_0) + \int e^{2\kappa \lambda_0 y + \kappa \hatsig_0}\,\Upsilon(\lambda_0 y) \vare^2_0.
\end{equation}

The bound on $\lambda$ in \eqref{conseq_of_BS1_strict_ineq} and since $\Upsilon(\lambda y) \leq \mathbf{1}_{y>(3\theta)^{-1}s_0}$, yields
\begin{equation}
    \begin{split}
        \int e^{2\kappa \lambda_0 y + \kappa \hatsig_0}\,\Upsilon(\lambda_0 y) \vare^2_0 \lesssim s^{\kappa}_0 \int_{y>(2\theta \lambda_0)^{-1}} e^{2\kappa \lambda_0 y} \vare^2_0 \lesssim s^{\kappa}_0 \int_{y>(3\theta)^{-1}s_0}e^{3\kappa s^{-1}_0 y}\vare^2_0 \lesssim s^{\kappa}_0\int_{y>(3\theta)^{-1}s_0}e^{y}\vare^2_0.
    \end{split}
\end{equation}
By interpolation and \eqref{conditons_ortho_and_decroissance_on_vare_0}, it holds
\begin{equation}
    s^j_0\mathcal{F}(s_0) \lesssim s^{j}_0\|\vare\|^2_{H^1} + s^j_0\int_{y>0}y\,\vare^2_0 \lesssim s^{-1}_0.
\end{equation}
Thus, by the coercivity of $\mathcal{F}$ in \eqref{coercivity_of_mathcal_F}, it holds
\begin{equation}
    \mathcal{N}^2_B \lesssim \mathcal{F}(s) \lesssim s^{-j} \mathcal{H}(s) \lesssim s^{-\frac 12}_0\,s^{-j}\leq \frac 12 s^{-j}.
\end{equation}
Therefore, the control of the local norm $\mathcal{N}^2_B$ is strictly improved.

\vspace{0.4cm}
The $H^1$ norm is improved by the estimates \eqref{estimate_on_L2_norm_of_vare} and \eqref{estimate_on_L2_norm_of_grad_vare} (choosing $s_0$ and $x_0$ to be sufficiently large).

\vspace{0.2cm}
\textit{Contradiction by a continuity argument.}\\
We construct a map $\Phi: [\lambda^-_0, \lambda^+_0] \to \{-1,1\}$ and we aim to prove its continuity and that $\Phi(\lambda^{\pm}_0) = \pm 1$, which would complete the contradiction argument.

We consider a continuous linear transformation $\lambda_0 \in [\lambda^-_0, \lambda^+_0] \mapsto \mu_0 \in [-1,1]$, such that $\mu_0(\lambda^{\pm}_0) = \pm 1$. Since $s^*$ is finite, we consider a map $s^* \in [s_0, +\infty) \mapsto h(s^*)(s^*)^{\frac 12 + 4\rho}\in \{-1,1\}$, which is also continuous.

It remains to show that $\mu_0(\lambda_0) \mapsto s^*$ is continuous. Define $G(s) = h^2(s)s^{1+8\rho}$. Let $s_1 \in [s_0,s^*]$ be such that $G(s_1) = 1$. The estimate \eqref{conseq_of_BS_on_g_and_hs} and the bound on $\rho$ in \eqref{rho_condition} yield the following transversality property for $s_0(\rho)\gg1$
\begin{equation}
\begin{split}
    G'(s_1) 
    &= 2h_s(s_1)h(s_1)s_1^{1+8\rho} + (1+8\rho)h^2(s_1)s_1^{8\rho}
    = \pm 2 h_s(s_1)s_1^{\frac 12 + 4\rho} + (1+8\rho)s^{-1}_1\,G(s_1)\\
    &\geq -C\,s_1^{-1-4\big(\frac{1}{16}-\rho\big)}+(1+8\rho)s^{-1}_1 \geq -Cs^{-4\big(\frac{1}{16}-\rho\big)}_0\,s^{-1}_1+(1+8\rho)s^{-1}_1 >\frac{1+8\rho}{2}s^{-1}_1 >0.
\end{split}
\end{equation}
In particular, $G(s^*)=1$.
Fix $\mu_0 \in (-1,1)$, therefore $s^* > s_0$. 
By continuity of $G$ and $G'$ we know that for some $0<\epsilon <s^*-s_0$ and $\delta_\epsilon>0$ sufficiently small, it holds $G(s^*+\epsilon)> 1+\delta$ and $G(s)<1-\delta$ for all $s \in [s_0,s^*-\epsilon]$. 

By continuous dependence on initial data for parameters, there exists $\epsilon_\mu$, such that 
\begin{equation}
    \lvert \mu_0 - \tilde{\mu}_0\rvert\leq \epsilon_\mu \Rightarrow \lvert G(s) - \tilde{G}(s) \rvert \leq \frac{\delta}{2} \qquad 
    \text{on }\;[s_0,s^*+\epsilon].
\end{equation}
Let $\tilde{G}$ and $\tilde{s}^*$ be the quantities associated with $\tilde{\mu}_0$. 
We get 
\begin{equation*}
    \forall s \in [s_0,s^*-\epsilon], \quad  \tilde{G}(s) \leq G(s)+\frac{\delta}{2} < 1-\frac{\delta}{2} \quad\Rightarrow\quad \tilde{s}^* \geq s^*-\epsilon
\end{equation*}
and
\begin{equation*}
    \tilde{G}(s^*+\epsilon)\geq G(s^*+\epsilon)-\frac{\delta}{2} >1+\frac{\delta}{2} \quad\Rightarrow\quad \tilde{s}^* \leq s^*+\epsilon.
\end{equation*}
This proves the continuity of $\Phi$ for $\mu_0 \in (-1,1)$. 
In the case $\mu_0 = \pm1$, we get $s^* = s_0$. By a similar argument, we get the continuity of $\Phi$ for $\mu_0 = \pm1$. 
Finally, since we have strictly improved \eqref{BS1} and \eqref{BS2}, the assumption \eqref{BS3} is saturated in $s^*$. Therefore, $s^*(\lambda^{\pm}_0) = s_0$ and $\Phi(\lambda^{\pm}_0) = \pm 1$ respectively. This leads to a contradiction. 

\end{proof}

\vspace{0.4cm}
\begin{proof}[End of the proof of the Theorem \ref{Theorem_principal_result}:]
Note that
\begin{equation}
    \|\partial_x U(t)\|_{L^2_x } = \lambda^{-1}(s)\|Q'\|_{L^2} + o_{t\to T}(\lambda^{-1}(s)).
\end{equation}

Indeed, we have
\begin{equation}
    \Big\| \partial_x\Big( U(t,x) - \lambda^{-\frac12}Q(y)  \Big) \Big\|_{L^2_x} \leq 
    \lambda^{-\frac 32}\,  b \|P'_b(y)\|_{L^2_x} + \lambda^{-\frac32}\| (\partial_x \vare)(s,y) \|_{L^2_x} + \|\partial_x f(t,x)\|_{L^2_x},
\end{equation}
where $y = \frac{x-\sigma(s)}{\lambda(s)}$ and $s = s(t)$.

Therefore, the definition of $P_b$, Lemma \ref{lemma_on_the_decaying_tail} and the estimate \eqref{estimate_on_L2_norm_of_grad_vare} yields
\begin{equation}
    \Big| \|\partial_x U(t)\|_{L^2_x} - \lambda^{-1}\|Q'\|_{L^2} \Big|\lesssim 
    \lambda^{-1}\,s^{-\frac 58} + |E(U_0)|^{\frac 12} + \delta(x^{-1}_0).
\end{equation}

From the definition of the rescaled variable and \eqref{BS2}, we have
    \begin{equation}
        T-t = \int^{+\infty}_{s(t)} \frac{d\tau(s')}{ds'}\, ds' = \int^{+\infty}_{s(t)}\lambda^{3}(s')\,ds' \sim \frac 12 s^{-2} \quad \text{and}\quad \lambda(s) \sim \big(2(T-t)\big)^{\frac 12}
    \end{equation}
Hence,
\begin{equation}
    \|\partial_x U(t)\|_{L^2} \sim \lambda^{-1}(s)\,\|Q'\|_{L^2} \sim c\, (T-t)^{-\frac 12}.
\end{equation}
Moreover, since $e^{\hatsig} \sim s$ and that $s^{-1} \sim \sqrt{2(T-t)}$, we have for $t \uparrow T$
    \begin{equation}
        \sigma(s) \sim \frac 12 \ln\Big( \frac{1}{2(T-t)}\Big) + \frac t4.
    \end{equation}

This concludes the proof of the Theorem \ref{Theorem_principal_result}.
\end{proof}